\documentclass[a4paper,leqno]{amsart}

\usepackage{amsmath,amssymb,amsthm,amsfonts}
\usepackage[T1]{fontenc}
\usepackage[ruled,vlined,commentsnumbered]{algorithm2e}
\usepackage{xspace}
\usepackage{url}
\usepackage{enumerate}
\usepackage{hyperref}
\usepackage{fullpage}
\usepackage{subcaption}
\usepackage{graphicx}

\theoremstyle{plain}
\newtheorem{theorem}{Theorem}
\newtheorem{lemma}[theorem]{Lemma}

\newtheorem{proposition}[theorem]{Proposition}
\newtheorem{corollary}[theorem]{Corollary}

\newtheorem{claim}[theorem]{Claim}

\theoremstyle{remark}
\newtheorem*{remark}{Remark}
\newtheorem*{question}{Question}
\theoremstyle{plain}

\newcommand{\pimi}[1][]{\ifthenelse{\equal{#1}{}}{\pi^{\circ}}{\pi_{#1}^{\circ}}}
\newcommand{\sigm}[1][]{\ifthenelse{\equal{#1}{}}{\sigma^{\circ}}{\sigma_{#1}^{\circ}}}
\newcommand{\rhom}[1][]{\ifthenelse{\equal{#1}{}}{\rho^{\circ}}{\rho_{#1}^{\circ}}}

\newcommand{\N}{\ensuremath{\mathbb{N}\xspace}}
\newcommand{\Z}{\ensuremath{\mathbb{Z}\xspace}}
\newcommand{\R}{\ensuremath{\mathbb{R}\xspace}}
\newcommand{\cC}{\ensuremath{\mathcal{C}\xspace}}

\newcommand{\bx}{\ensuremath{\mathbf{x}}\xspace}
\newcommand{\bzero}{\ensuremath{\mathbf{0}}\xspace}

\newcommand{\cD}{\ensuremath{\mathcal{D}}\xspace}
\newcommand{\cE}{\ensuremath{\mathcal{E}}\xspace}

\newcommand{\cH}{\ensuremath{\mathcal{H}}\xspace}

\DeclareMathOperator{\Maj}{Maj}
\DeclareMathOperator{\Supp}{supp}

\DeclareMathOperator*{\Med}{Med}

\DeclareMathOperator{\Ma}{Ma}
\DeclareMathOperator{\Me}{Me}
\DeclareMathOperator{\dM}{M}

\begin{document}

\title{ABC(T)-graphs: an axiomatic characterization of the median
  procedure in graphs with connected and G$^2$-connected medians}

\author[L. B\'en\'eteau]{Laurine B\'en\'eteau} \address{Aix-Marseille Universit\'e and
  CNRS, LIS, Marseille, France} \email{laurine.beneteau@lis-lab.fr}

\author[J.\ Chalopin]{J\' er\'emie Chalopin} \address{CNRS and
  Aix-Marseille Universit\'e, LIS, Marseille, France}
\email{jeremie.chalopin@lis-lab.fr}

\author[V. Chepoi]{Victor Chepoi} \address{Aix-Marseille Universit\'e
  and CNRS, LIS, Marseille, France} \email{victor.chepoi@lis-lab.fr}

\author[Y. Vax\`es]{Yann Vax\`es} \address{Aix-Marseille Universit\'e
  and CNRS, LIS, Marseille, France} \email{yann.vaxes@lis-lab.fr}

\begin{abstract}
  The median function is a location/consensus function that maps any
  profile $\pi$ (a finite multiset of vertices) to the set of vertices
  that minimize the distance sum to vertices from $\pi$. The median
  function satisfies several simple axioms: Anonymity (A), Betweeness
  (B), and Consistency (C).  McMorris, Mulder, Novick and Powers
  (2015) defined the ABC-problem for consensus functions on graphs as
  the problem of characterizing the graphs (called, ABC-graphs) for
  which the unique consensus function satisfying the axioms (A), (B),
  and (C) is the median function.
In this paper, we show that modular graphs with $G^2$-connected
  medians (in particular, bipartite Helly graphs) are ABC-graphs. On
  the other hand, the addition of some simple local axioms satisfied
  by the median function in all graphs (axioms (T), and (T$_2$))
  enables us to show that all graphs with connected median (comprising
  Helly graphs, median graphs, basis graphs of matroids and even
  $\Delta$-matroids) are ABCT-graphs and that benzenoid graphs are
  ABCT$_2$-graphs. McMorris et al (2015) proved that the graphs
  satisfying the pairing property (called the intersecting-interval
  property in their paper) are ABC-graphs. We prove that graphs with
  the pairing property constitute a proper subclass of bipartite Helly
  graphs and we discuss the complexity status of the recognition
  problem of such graphs.
\end{abstract}

\maketitle

\section{Introduction}
The median problem is one of the oldest optimization problems in
Euclidean geometry~\cite{LoMoWe}. Although it can be formulated for
all metric spaces, we will consider it for graphs.
Given a (not necessarily finite) connected graph $G=(V,E)$, a
\emph{profile} is any finite sequence $\pi = (x_1, \ldots, x_n)$ of
vertices of $G$.  The \emph{total distance} of a vertex $v$ of $G$ is
defined by $F_{\pi}(v)=\sum_{i=1}^n d(v,x_i)$. A vertex $v$ minimizing
$F_{\pi}(v)$ is called a \emph{median vertex} of $G$ with respect to
$\pi$ and the set of all medians vertices is called the \emph{median
  set} $\Med_G(\pi)$.  Finally, the mapping that associates to any
profile $\pi$ of $G$ its median set $\Med_G(\pi)$ is called the
\emph{median function}. Consequently, the median function
$\Med=\Med_G$ can be viewed as a particular consensus function. In the
consensus problem in social group choice, given individual preferences
of candidates one has to compute a consensual group decision that best
reflects those preferences. It is usually required that the consensus
respects some simple axioms ensuring that it remains reasonable and
rational.  However, by the classical Arrow's~\cite{Arr} impossibility
theorem, there is no consensus function satisfying natural
``fairness'' axioms.  In this respect, the Kemeny
median~\cite{Ke,KeSn} is an important consensus function satisfying
most of fairness axioms.  It corresponds to the median problem in
graphs of permutahedra. 

Following Arrow's axiomatic approach to consensus functions, one may
wish to characterize axiomatically a consensus function $M$ of a given
type among all consensus functions. Then the goal is either to
characterize $M$ by a set of axioms or to characterize the instances
for which $M$ is the unique consensus function satisfying a set of
natural axioms.  Holzman~\cite{Holzman} was the first to study a
location function axiomatically as a consensus problem.  Namely, he
characterized the barycenter function of trees, namely the function
$B$ mapping each profile $\pi=(x_1,\ldots,x_n)$ to the set of vertices
minimizing the mean $F_{\pi}(v)=\sum_{i=1}^n d^2(v,x_i)$. Notice also
that the axiomatic characterization of the barycentric function on the
line or in ${\mathbb R}^d$ is a classical problem~\cite{Acz}, having
its origins in the paper~\cite{Ko} by Kolmogorov.  Foster and
Vohra~\cite{FoVo} axiomatically characterized the median function
$\Med$ on continuous trees (where the location can be on edges), and
McMorris, Mulder and Roberts~\cite{McMoMuRo} characterized $\Med$ in
the case of combinatorial trees.  They also showed that three basic
axioms (Anonymity (A), Betweenness (B) and Consistency (C)) are
sufficient to characterize the median function of cube-free median
graphs.  Median graphs constitute the most important class of graphs
in metric graph theory~\cite{BaCh_survey}.  Those are the graphs in
which for each triplet of vertices there exists exactly one vertex
lying on shortest paths between any pair of vertices of the
triplet. The nice result of Mulder and Novick~\cite{MuNo} shows that
in fact the axioms (A), (B), and (C) characterize $\Med$ in all median
graphs. They called the consensus functions that respect the axioms
(A), (B), and (C) \emph{ABC-functions}.  Recently, McMorris, Mulder,
Novick, and Powers~\cite{McMoMuNoPo} introduced the \emph{ABC-problem}
as the characterization of graphs (which we call \emph{ABC-graphs})
with unique ABC-functions (which then necessarily coincide with
$\Med$).
Additionally to median graphs, they proved that the graphs satisfying
the intersecting-intervals property (which we call \emph{pairing
  property}) are ABC-graphs. On the other hand, they showed that the
complete graph $K_n$ admits at least two ABC-functions for any
$n \ge 3$.  Finally, they asked to find additional axioms that hold
for the median function on all graphs and which added to (A), (B), (C)
characterize $\Med$.

Bandelt and Chepoi~\cite{BaCh_median} characterized the \emph{graphs
  with connected medians}, i.e., the graphs in which all median sets
induce connected subgraphs (equivalently, for all profiles $\pi$, the
function $F_{\pi}$ is unimodal).  They also established that
several important classes of graphs (Helly graphs, basis graphs of
matroids, and weakly median graphs) are graphs with connected
medians. Recently, in~\cite{GpConMed} we generalized these results and
characterized the graphs with $G^p$-connected medians, i.e., the
graphs in which all median sets induce connected subgraphs in the
$p$th power $G^p$ of the graph $G$. We also established that some
important classes of graphs have $G^2$-connected medians (chordal
graphs, bridged graphs, graphs with convex balls, bipartite Helly
graphs, benzenoids).

In this paper, we continue the research on the ABC-problem in several
directions.  First, we prove that modular graphs with $G^2$-connected
medians (and in particular, bipartite Helly graphs) are
ABC-graphs. Second, we prove that graphs with connected medians are
ABCT-graphs, i.e., graphs in which the median function is the unique
consensus function satisfying (A), (B), (C), and a new axiom (T),
requiring that the consensus function returns all the elements of the
profile when the profile is a triplet of three pairwise adjacent
vertices. We extend (T) to a similar axiom (T$_2$) on triplets of
vertices at pairwise distance 2 and prove that benzenoids are
ABCT$_2$-graphs (but we show that the 6-cycle, which is the simplest
benzenoid, has at least two ABC-functions). Finally, we prove that the
graphs with the pairing or double-pairing property constitute a proper
subclass of bipartite Helly graphs. We characterize the graphs with
the pairing/double-pairing property as bipartite Helly graphs
satisfying a local condition on balls of radius 2.  We show that the
problem of deciding if a bipartite Helly graph satisfies the double
pairing property is in co-NP. However we do not know if one can decide
in polynomial time (or even in non-deterministic polynomial time) if a
bipartite Helly graph satisfies the double-pairing or the pairing
property.

\section{Preliminaries}\label{sec:def}

\subsection{Graphs}
All graphs $G = (V, E)$ in this paper are undirected, simple, and
connected; $V$ is the vertex-set and $E$ is the edge-set of $G$. We
write $u \sim v$ if $u, v \in V$ are adjacent.  Given a graph
$G=(V,E)$, the \emph{distance} $d_G(u,v)$ between two vertices $u$ and
$v$ is the length of a shortest $(u,v)$-path. If there is no
ambiguity, we denote $d(u,v)=d_G(u,v)$. For a vertex $v$ and a set
$A\subset V$, we denote by $d(v,A)=\min\{ d(v,x): x\in A\}$ the
distance from $v$ to the set $A$.  The \emph{interval} $I(u,v)$
between two vertices $u$ and $v$ is the set of all vertices on
shortest $(u,v)$-paths, i.e.  $I(u,v)=\{w: d(u,w)+d(w,v)=d(u,v)\}$. We
denote by $I^{\circ}(u,v)=I(u,v)\setminus \{ u,v\}$ the ``interior''
of the interval $I(u,v)$. We say that a pair of vertices $(u,v)$ of
$G$ is a \emph{2-pair} if $d(u,v)=2$.  An induced subgraph $H$ (or the
corresponding vertex-set of $H$) of a graph $G$ is \emph{gated} if for
every vertex $x$ outside $H$ there exists a vertex $x'$ in $H$ (the
\emph{gate} of $x$) such that $x'\in I(x,y)$ for any $y$ of $H$. The
gate $x'$ of $x$ is necessarily unique.
A \emph{ball} of radius $r$ and center $v$ is the set $B_r(v)$ of all
vertices $x$ of $G$ such that $d(v,x)\leq r$. A \emph{half-ball} of a
bipartite graph $G=(X\cup Y,E)$ is the intersection of a ball of $G$
with one of the two color classes $X$ or $Y$. By $N(v)$ and $N(A)$ we
denote the neighborhood of a vertex $v$ or a set $A\subset V$.  Let
$H$ be a subgraph of a graph $G$. We call $H$ an \emph{isometric
  subgraph} of $G$ if $d_H(u,v)=d_G(u,v)$ for each pair of vertices
$u,v$ of $H$.

Vertices $v_1,v_2,v_3$ of a graph $G$ form a \emph{metric triangle}
$v_1v_2v_3$ if the intervals $I(v_1,v_2), I(v_2,v_3)$, and
$I(v_3,v_1)$ pairwise intersect only in the common end-vertices, i.e.,
$I(v_i, v_j) \cap I(v_i,v_k) = \{v_i\}$ for any
$1 \leq i, j, k \leq 3$.
If $d(v_1,v_2)=d(v_2,v_3)=d(v_3,v_1)=k$, then this metric triangle is
called \emph{equilateral} of \emph{size} $k$. An equilateral metric triangle
$uvw$ of size $k$ is called \emph{strongly equilateral} if $d(v_3,x)=k$ for any vertex $x\in I(v_1,v_2)$.
A metric triangle
$v_1v_2v_3$ of $G$ is a \emph{quasi-median} of the triplet $x,y,z$
if the following metric equalities are satisfied:
\begin{align*}
  d(x,y)&=d(x,v_1)+d(v_1,v_2)+d(v_2,y),\\
  d(y,z)&=d(y,v_2)+d(v_2,v_3)+d(v_3,z),\\
  d(z,x)&=d(z,v_3)+d(v_3,v_1)+d(v_1,x).
\end{align*}
If the size of $v_1v_2v_3$ is zero, which means that $v_1$, $v_2$, and $v_3$ are the same vertex $v$,
then this vertex $v$ is called a \emph{median} of $x,y,z$.
A median may not exist and may not be unique.
On the other hand, a quasi-median of every triplet $x,y,z$ always exists.

\begin{figure}[ht]
  \begin{center}
    \includegraphics{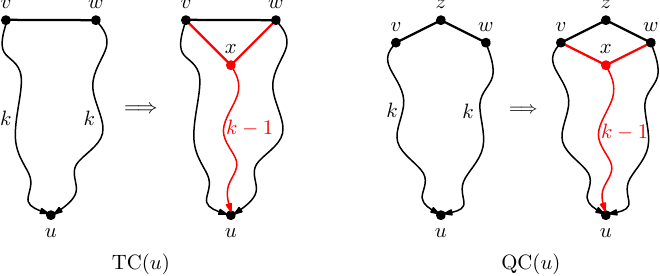}
  \end{center}
  \caption{Triangle and quadrangle conditions}\label{fig-conditions}
\end{figure}

\subsection{Classes of graphs}
In this subsection, we recall the definitions of some classes of graphs
that will be investigated in the paper.
A graph $G$ is \emph{weakly modular}~\cite{BaCh_helly,Ch_metric} if for any vertex $u$
its distance function $d$ satisfies the following triangle and quadrangle conditions
(see Figure~\ref{fig-conditions}):
\begin{itemize}
\item
  \emph{Triangle condition} TC($u$):  for any two vertices $v,w$ with
  $1=d(v,w)<d(u,v)=d(u,w)$ there exists a common neighbor $x$ of $v$
  and $w$ such that $d(u,x)=d(u,v)-1$.
\item
  \emph{Quadrangle condition} QC($u$): for any three vertices $v,w,z$ with
  $d(v,z)=d(w,z)=1$ and  $2=d(v,w)\leq d(u,v)=d(u,w)=d(u,z)-1$, there
  exists a common neighbor $x$ of $v$ and $w$ such that
  $d(u,x)=d(u,v)-1$.
\end{itemize}
All metric triangles of weakly modular graphs are equilateral. In
fact, a graph $G$ is weakly modular if and only if any metric triangle
$v_1v_2v_3$ of $G$ is strongly equilateral~\cite{Ch_metric}, i.e., if
all quasi-medians of triplets are strongly equilateral metric
triangles. If all quasi-medians are medians, i.e., metric triangles of
size 0, then such weakly modular graphs are called
\emph{modular}.  Modular graphs in which every triplet of vertices has a unique
quasi-median are called \emph{median graphs}. Median graphs
constitute the most important main class of graphs in metric graph
theory~\cite{BaCh_survey}.  Analogously to median graphs, \emph{weakly
  median graphs} are the weakly modular graphs in which all triplets
have unique quasi-medians.

Another important subclass of weakly modular graphs is the class of
Helly graphs.  A graph $G$ is a \emph{Helly graph} if the family of
balls of $G$ has the Helly property, that is, every finite collection
of pairwise intersecting balls of $G$ has a nonempty intersection.
Analogously, a \emph{bipartite Helly graph} is a bipartite graph such that the
collection of half-balls of $G$ has the Helly property. Bipartite
Helly graphs are modular.

A graph $G=(V,E)$ is called \emph{meshed}~\cite{BaCh_median} if for
any three vertices $u,v,w$ with $d(v,w)=2$, there exists a common
neighbor $x$ of $v$ and $w$ such that $2d(u,x)\le
d(u,v)+d(u,w)$. Metric triangles of meshed graphs are
equilateral~\cite{BaCh_median}. It is well-known that meshed graphs
satisfy the triangle condition (see~\cite[Lemma 2.23]{CCHO} for
example) but they do not necessarily satisfy the quadrangle
condition. All weakly modular graphs and all graphs with connected
medians are meshed~\cite{BaCh_median}.  Another important subclasses
of meshed graphs with connected medians are basis graphs of
matroids~\cite{Mau} and of even $\Delta$-matroids~\cite{Che_bas}.

A \emph{hexagonal grid} is a grid formed by a tessellation of the plane
${\mathbb R}^2$
into regular hexagons.  Let $Z$ be a cycle of the hexagonal grid. A
\emph{benzenoid system} is a subgraph of the
hexagonal grid induced by the vertices located in the closed region of the
plane bounded by $Z$.  Equivalently, a benzenoid system is a finite connected
plane graph in which every interior face is a regular
hexagon of side length 1. Benzenoid systems play an important role in
chemical graph theory~\cite{GuCy}.

\subsection{Axioms} Given a graph $G=(V,E)$, we call a \emph{profile}
any sequence $\pi = (x_1, \ldots, x_n)$ of vertices of $G$ (notice
that the same vertex may occur several times in $\pi$ or not at
all). The set $\{ 1,\ldots, n\}$ can be interpreted as a set of agents
and for each agent $i\in [n]$, $x_i$ is called the \emph{location} of
$i$ in $G$. Often the profile $\pi = (x_1, \ldots, x_n)$ is referred
to as the \emph{location profile}.  Denote by $V^*$ the set of all
profiles of finite length.  For two profiles $\pi,\pi'\in V^*$, we
write $\pi'\subseteq \pi$ if $\pi'$ is a subprofile of $\pi$.  A
\emph{consensus function} on a graph $G=(V,E)$ is any function
$L:V^* \rightarrow 2^V \setminus \varnothing$. For two profiles $\pi$
and $\rho$, denote by $\pi \rho$ the concatenation of $\pi$ and
$\rho$.  For a positive integer $k$ and a profile $\pi$, denote by
$\pi^k$ the $k$ concatenations of the profile $\pi$. When $k=2$, we
say that the profile $\pi^2$ is a \emph{double-profile}.  For a vertex
$v$, the \emph{weight} $\pi(v)$ of $v$ is the number of occurrences of
$v$ in $\pi$.  We call a vertex $v$ a \emph{majority vertex} of $\pi$
if $\pi(v)\geq\frac{1}{2}|\pi|$. If $\pi$ has a majority vertex then
we say that $\pi$ is a \emph{majority profile}.  We say that a profile
is \emph{tie} if it has two majority vertices $u$ and $v$, i.e.,
$\pi(v)=\pi(u)=\frac{1}{2}|\pi|$.  Denote by $\Maj(\pi)$ the majority
vertices of a profile $\pi$ and by $\Supp(\pi)$ the \emph{support} of
a profile $\pi$, i.e., the vertices $v$ that appear in $\pi$.

Given a profile $\pi=(x_1, \ldots, x_n)$ and a vertex $v$ of $G$, let
\[
  F_{\pi}(v)=\sum_{i=1}^n d(v,x_i)=\sum_{x\in V} \pi(x)d(v,x).
\]
A vertex $v$ of $G$ minimizing $F_{\pi}(v)$ is called a \emph{median
  vertex} of $G$ with respect to $\pi$ and the set of all medians
vertices is called the \emph{median set} $\Med(\pi)$.
Finally, the mapping $\Med$ that associates to any profile $\pi$ of $G$
its median set $\Med(\pi)$ is called the \emph{median
  function}. Clearly, $\Med$ is a consensus function.

Motivated by the axiomatic approach of~\cite{BaMo}, McMorris, Mulder,
and Roberts considered the combination of the following simple axioms
(A), (B), (C) for consensus functions on graphs~\cite{McMoMuRo}. 

\begin{enumerate}[(A)]
\item \emph{Anonymity:} for any profile
  $\pi = (x_1, x_2, \ldots, x_n)$ and any permutation $\sigma$ of
  $\{1,2,\ldots,n\}$, $L(\pi) = L(\pi^{\sigma})$ where
  $\pi^{\sigma} = (x_{\sigma(1)}, x_{\sigma(2)}, \ldots,
  x_{\sigma(n)})$.
\item \emph{Betweeness:} $L(u,v) = I(u,v)$.
\item \emph{Consistency:} for any profiles $\pi$ and $\rho$, if
  $L(\pi)\cap L(\rho) \neq \varnothing$, then
  $L(\pi\rho) = L(\pi)\cap L(\rho)$.
\end{enumerate}

The Consistency axiom is known as the Young Consistency
condition~\cite{Young74,YoLe78} and in the context of the
median function, it was first considered by Barth\'elemy and
Monjardet~\cite{BaMo}. The combination of the axioms (A), (B), (C)
have been further studied for median functions in graphs
in~\cite{McMoMuNoPo,MuNo}. 

We also consider the following triangle axioms, that are meaningful
only for graphs with triangles:
\begin{enumerate}[(T)]
\item \emph{Triangle:} for any three pairwise adjacent vertices
  $u, v, w$ of $G$, $L(u,v,w) = \{u,v,w\}$.
\end{enumerate}
\begin{enumerate}[(T$^-$)]
\item \emph{weak Triangle:} for any three pairwise adjacent vertices
  $u, v, w$ of $G$, if $u \in L(u,v,w)$, then
  $\{u,v,w\} \subseteq L(u,v,w)$.
\end{enumerate}

One can easily see that the median function $\Med$ satisfies the
axioms (A), (B), (C), (T), and (T$^-$).  The authors
of~\cite{McMoMuNoPo} termed the consensus functions satisfying the
axioms (A), (B) and (C), \emph{ABC-functions}. We say that a graph $G$
is an \emph{ABC-graph} if there exists an unique ABC-function on $G$,
i.e., if the median function is the only ABC-function.  The
\emph{ABCT-} and \emph{ABCT$^-$-}functions and the \emph{ABCT-},
\emph{ABCT$^-$-}graphs are defined in a similar way. Characterizing
the ABC-graphs (or the ABCT-graphs) is an open
problem. The first problem is referred in~\cite{McMoMuNoPo} as to the
ABC-\emph{problem} and we will refer to the other as to the
ABCT-\emph{problem}.

\subsection{Graphs with connected and $G^p$-connected medians}
In this subsection we recall some definitions and results
from~\cite{BaCh_median} and~\cite{GpConMed}. The $pth$ power of a
graph $G=(V,E)$ is a graph $G^p$ having the same vertex-set $V$ as $G$
and two vertices $u$ and $v$ are adjacent in $G^p$ if and only if
$d_G(u,v) \leq p$. Let $f$ be a real-valued function defined on the
vertex set $V$ of $G$. We say that $f$ is \emph{unimodal} on $G^p$ if
any local minimum of $f$ in $G^p$ is a global minimum of $f$, i.e., if
for a vertex $u$ we have $f(u)\le f(v)$ for any $v\in V$ such that
$d_G(u,v)\le p$, then $f(u)\le f(v)$ for any $v\in V$. Analogously, we
say that the function $f$ has \emph{$G^p$-connected} (respectively,
\emph{$G^p$-isometric}) minima, if the set of minima of $f$ induce a
connected (respectively, isometric) subgraph of $G^p$.  Finally, we
say that $G$ is a graph with \emph{$G^p$-connected} (respectively,
\emph{$G^p$-isometric}) \emph{medians} if for any profile $\pi$, the
function $F_{\pi}$ is a function with $G^p$-connected (respectively,
$G^p$-isometric) minima. If $p=1$, we say that $G$ has
connected/isometric medians.

A \emph{$p$-geodesic} between two vertices $u,v$ is a finite sequence
of vertices $P=(u=w_0,w_1,\ldots ,w_n=v)$ included in a
$(u,v)$-geodesic of $G$ such that $d_G(w_i,w_{i+1})\le p$ for any
$i=0,\ldots,n-1$.
A function $f$ defined on the vertex-set $V$ of a graph $G$ is called
\emph{$p$-weakly peakless} if any pair of vertices $u,v$ of $G$ can be
connected by a $p$-geodesic $P=(u=w_0,w_1,\ldots,w_{n-1},w_n=v)$ along
which $f$ is peakless, i.e., if $0\leq i<j<k\leq p$ implies
$f(w_j)\leq \max\{ f(w_i),f(w_k)\}$ and equality holds only if
$f(w_i)=f(w_j)=f(w_k)$. For $p=1$, the $p$-weakly peakless functions
are called \emph{pseudopeakless}~\cite{BaCh_median}. Generalizing a
result for pseudopeakless functions of~\cite{BaCh_median}, it was
shown in~\cite{GpConMed} that a function $f$ is $p$-weakly peakless if
and only if it is locally-$p$-weakly peakless, i.e., for any $u,v$
such that $p+1 \leq d(u,v) \leq 2p$ there exists
$w \in I^{\circ}(u,v)$ such that
$f(w) \leq \max \left\{f(u),f(v)\right\}$, and equality holds only if
$f(u) = f(w) =f(v)$.

Generalizing a result from~\cite{BaCh_median} for $p=1$,
in~\cite{GpConMed} we characterized graphs with $G^p$-connected
medians in the following way:

\begin{theorem}[\cite{GpConMed}]\label{th-cmed-p}
  For a graph $G$ and an integer $p\ge 1$, the following conditions
  are equivalent:
  \begin{enumerate}
  \item\label{th-cmedp-0} $F_{\pi}$ is unimodal on $G^p$ for any
    profile $\pi$;
\item\label{th-cmedp-3} $F_{\pi}$ is $p$-weakly peakless for any
    profile $\pi$;
\item\label{th-cmedp-5} $G$ is a graph with $G^p$-isometric medians;
  \item\label{th-cmedp-6} $G$ is a graph with $G^p$-connected medians.
  \end{enumerate}
\end{theorem}

In~\cite{BaCh_median}, it was established that the graphs with
connected medians are meshed. Consequently, they satisfy the triangle
condition TC and their metric triangles are equilateral.

It was established in~\cite{BaCh_median} and~\cite{GpConMed} that many
important classes of graphs have connected or
$G^2$-connected medians. In particular, Helly graphs, weakly median
and median graphs, basis graphs of matroids and of even
$\Delta$-matroids have connected medians~\cite{BaCh_median}, and
bipartite Helly graphs and benzenoids have $G^2$-connected
medians~\cite{GpConMed}.  The results of our paper concern all graphs
with connected medians (including the classes mentioned above), all
modular graphs with $G^2$-connected medians (in particular bipartite
Helly graphs), and benzenoids.

\section{ABC(T)-graphs}\label{sec:charact}
In this section, we prove two results: (1) that graphs with connected
medians are ABCT-graphs and (2) that modular graphs with
$G^2$-connected medians are ABC-graphs.

\subsection{Properties of ABC-functions}
We continue with some useful properties of ABC-functions.

\begin{lemma}\label{lem-xinL}
  Let $L$ be an ABC-function on a graph $G$. Then for any profile
  $\pi$, any vertex $x \in V$, and any $x' \in L(\pi,x)$, we have
  $x' \in L(\pi,x')\subseteq L(\pi,x)$.
\end{lemma}

\begin{proof}
  Pick any $x''\in L(\pi,x')$ and consider the profile
  $\pi''= \pi,x,x',x''$. Since $x' \in L(\pi,x)$ and
  $x' \in L(x',x'') = I(x',x'')$, we have
  $x' \in L(\pi'') = L(\pi,x) \cap L(x',x'')$. Since
  $x'' \in L(\pi,x')$ and $x'' \in L(x,x'') = I(x,x'')$, we have
  $x'' \in L(\pi'') = L(\pi,x') \cap L(x,x'')$. Consequently,
  $x' \in L(\pi'') \subseteq L(\pi,x')$ and
  $x'' \in L(\pi'') \subseteq L(\pi,x)$, establishing
  $x' \in L(\pi,x')\subseteq L(\pi,x)$.
\end{proof}

Lemma~\ref{lem-xinL} implies that if $x\in L(u,v,w)$, then
$x \in L(u,v,x)$ and $L(u,v,x) \subseteq L(u,v,w)$.

\begin{lemma}\label{lem-inter-interval}
  Let $L$ be an ABC-function on a graph $G$. Then for any profile
  $\pi$ with $|\pi| \geq 2$ and any vertex $x \in V$ such that
  $x \in L(\pi,x)$, we have $\bigcap_{y \in \pi} I(x,y) = \{x\}$.
\end{lemma}

\begin{proof}
  Pick $z \in \bigcap_{y \in \pi} I(x,y)$, let $k = |\pi|$, and
  consider the profile $\pi' = \pi,x^k,z^{k-1}$. Since
  $x \in L(\pi,x)$ and $x \in L(x^{k-1},z^{k-1})= L(x,z) = I(x,z)$,
  $x \in L(\pi') = L(\pi,x) \cap I(x,z)$. Since
  $z \in \bigcap_{y \in \pi} L(y,x) = \bigcap_{y \in \pi} I(y,x)$ and
  $z \in L(z^k) = L(z) = \{z\}$,
  $L(\pi') = \bigcap_{y \in \pi} I(y,x) \cap \{z\} = \{z\}$ and
  consequently $x = z$.
\end{proof}

Lemma~\ref{lem-inter-interval} implies that if
$\{u,v,w\} \subseteq L(u,v,w)$, then $u,v,w$ form a metric triangle,
i.e., the intervals $I(u,v)$, $I(u,w)$ and $I(v,w)$ pairwise intersect
only on the common end vertices.

\subsection{Graphs with connected medians are ABCT-graphs}

In this section, we show that the graphs in which the median set of
any profile is connected satisfy the ABCT-property.  This generalizes
the result of~\cite{MuNo} for median graphs that are precisely the
bipartite graphs in which the median set of any profile is connected.
The following lemma establishes that for the ABC-functions in graphs
with connected medians, (T$^-$) is equivalent to (T):

\begin{lemma}\label{lem:TC-T-E1}
  If $G$ is a graph satisfying the triangle condition TC and $L$ is an
  ABCT$^-$-function, then $L$ also satisfies axiom (T).
\end{lemma}

\begin{proof}
  Consider three pairwise adjacent vertices $u,v,w$ and assume that
  there exists $x \in L(u,v,w)\setminus \{u,v,w\}$. Since
  $x \in L(u,v,w)$, by Lemma~\ref{lem-xinL}, we have $x \in L(u,v,x)$
  and consequently, by Lemma~\ref{lem-inter-interval},
  $I(u,x) \cap I(v,x) = \{x\}$. Since $G$ satisfies the triangle
  condition, necessarily we get $x \sim u,v$.  Since $x \in L(u,v,x)$,
  by (T$^-$), we get that $u,v,x \in L(u,v,x)$. Consequently,
  $u \in L(u,v,x) \cap L(u,w) = L(u,v,w,x,u)$ and since
  $L(u,w) = \{u,w\}$, we get that $L(u,v,w,x,u) \subseteq
  \{u,w\}$. But we have $x \in L(u,v,w) \cap L(u,x) = L(u,v,w,x,u)$,
  leading to a contradiction. Consequently,
  $L(u,v,w) \subseteq \{u,v,w\}$ and by (T$^-$), we have
  $L(u,v,w) = \{u,v,w\}$.
\end{proof}

In the following, we consider graphs satisfying the triangle condition
TC and ABCT-functions on these graphs. We generalize some ideas of
Mulder and Novick~\cite{MuNo} to establish that if $G$ has connected
medians, then $L$ is the median function on $G$.  Since median graphs
have connected medians and do not contain triangles, this provides a
new proof of the main result of~\cite{MuNo} which does not use
specific properties of median graphs other than the fact that in
median graphs, the median set of any profile is connected.

For an edge $uv$ of $G=(V,E)$, let
$W_{uv} = \{x \in V: d(u,x) < d(v,x)\}$,
$W_{vu} = \{x \in V: d(v,x) < d(u,x)\}$, and
$W_{uv}^{=} = \{x \in V: d(u,x) = d(v,x)\}$. Note that if $G$ is
bipartite, then $W_{uv}^{=} = \varnothing$. Given a profile $\pi$, we
denote by $\pi_{uv}$ (respectively, $\pi_{vu}, \pi_{uv}^=$) the
restriction of $\pi$ to $W_{uv}$ (respectively, $E_{vu}, E_{uv}^=$).
The following observation is immediate:

\begin{lemma}\label{obs-Lvx}
  For any edge $uv$ of a graph $G$ and any $x \in W_{uv}$ and
  $y \in W_{vu}$, we have $u,v \in L(v,x) = I(v,x)$ and
  $u,v \in L(u,y) = I(u,y)$.
\end{lemma}

\begin{lemma}\label{lem-Wuveq}
  Let $G$ be a graph satisfying the triangle condition TC. For any
  edge $uv$ of $G$ and any $z \in W_{uv}^{=}$, we have
  $u,v \in L(u,v,z)$.
\end{lemma}

\begin{proof}
  By TC($z$), there exists $w \sim u,v$ such that
  $d(w,z) = d(u,z) -1 = d(v,z) -1$.

  \begin{claim}\label{claim-wLuvz}
    Either $u,v,w\in L(u,v,z)$ or $u,v,w \notin L(u,v,z)$.
  \end{claim}

  \begin{proof}
    By symmetry, it suffices to show that $u \in L(u,v,z)$ if and only
    if $w \in L(u,v,z)$.  Suppose that
    $\{u,w\} \cap L(u,v,z) \neq \varnothing$ and consider the profile
    $(u,v,z,u,w)$. Since $L(u,w) = \{u,w\}$ and
    $\{u,w\} \cap L(u,v,z) \neq \varnothing$, we have
    $L(u,v,z,u,w) = \{u,w\} \cap L(u,v,z)$.  By (T),
    $L(u,v,w) = \{u,v,w\}$ and since $L(u,z) = I(u,z)$, we have
    $u,w \in L(u,v,w) \cap L(u,z) = L(u,v,z,u,w) = \{u,w\} \cap
    L(u,v,z)$. Consequently, $u \in L(u,v,z)$ iff $w \in L(u,v,z)$.
This ends the proof of the claim.
  \end{proof}

  If $u \in L(u,v,z)$ or $v \in L(u,v,z)$, we are done by
  Claim~\ref{claim-wLuvz}. Suppose now that $u,v \notin L(u,v,z)$ and
  let $z' \in L(u,v,z)$. By Lemma~\ref{lem-xinL},
  $z' \in L(u,v,z') \subseteq L(u,v,z)$ and by
  Lemma~\ref{lem-inter-interval}, $I(u,z') \cap I(v,z') =
  \{z'\}$. Since $G$ satisfies TC($z'$), necessarily, $z' \sim
  u,v$. Therefore, by (T), $L(u,v,z') = \{u,v,z'\}$ and consequently,
  $u,v \in L(u,v,z') \subseteq L(u,v,z)$. This ends the proof of the
  lemma.
\end{proof}

Given a graph $G=(V,E)$ and a profile $\pi\in V^*$, let $\pi_{uv}$ be
the restriction of $\pi$ to the set $W_{uv}$, $\pi_{vu}$ be the
restriction of $\pi$ to $W_{vu}$, and $\pi_{uv}^{=}$ be the
restriction of $\pi$ to $W_{uv}^{=}$.

\begin{lemma}\label{lem-Lpiprime}
  Let $G=(V,E)$ be a graph satisfying the triangle condition TC and
  $\pi\in V^*$ be any profile. For any edge $uv$ of $G$, we have
  $u,v \in L(\pi')$ where $\pi' = \pi,u^{l+p},v^{k+p}$ with
  $l = |\pi_{vu}|$, $k = |\pi_{uv}|$, and $p = |\pi_{uv}^{=}|$.
\end{lemma}

\begin{proof}
  Let $\pi_{uv} = (x_1,\ldots,x_k)$, $\pi_{vu} = (y_1,\ldots,y_{l})$,
  and $\pi_{uv}^{=} = (z_1,\ldots,z_p)$. Consider the profile
  $\pi' = \pi,u^{l+p},v^{k+p}$.  Note that
  $u,v \in \bigcap_{i=1}^k L(v,x_i) = \bigcap_{i=1}^k I(v,x_i)$ and
  that $u,v \in \bigcap_{i=1}^l L(u,y_i) = \bigcap_{i=1}^k I(u,y_i)$
  by Lemma~\ref{obs-Lvx}. Note also that
  $u,v \in \bigcap_{i=1}^p L(u,v,z_i)$ by Lemma~\ref{lem-Wuveq}.
  Consequently,
  $u,v \in L(v,x_1, \ldots, v,x_k, u,y_1, \ldots, u, y_l, u, v, z_1,
  \ldots, u, v, z_p) = L(x_1, \ldots, x_k, y_1, \ldots, y_l, z_1,
  \ldots, z_p, u^{l+p}, v^{k+p}) = L(\pi')$.
\end{proof}

\begin{lemma}\label{lem-L-wconv}
  Let $G=(V,E)$ be a graph satisfying the triangle condition TC and
  $\pi\in V^*$ be any profile.  For any edge $uv$ of $G$ the following
  holds:
  \begin{enumerate}[(1)]
  \item  If $F_{G,\pi}(u) = F_{G,\pi}(v)$, then $u \in L(\pi)$ iff
    $v \in L(\pi)$.
  \item  If $F_{G,\pi}(u) > F_{G,\pi}(v)$, then $u \notin L(\pi)$.
  \end{enumerate}
\end{lemma}

\begin{proof}
  Let $l = |\pi_{vu}|$, $k = |\pi_{uv}|$, and $p = |\pi_{uv}^{=}|$.
  Assume that $F_{G,\pi}(u) \geq F_{G,\pi}(v)$. Observe that
  $F_{G,\pi}(u) - F_{G,\pi}(v) = |\pi_{vu}| - |\pi_{uv}| = l - k \geq
  0$.  As in Lemma~\ref{lem-Lpiprime}, we consider the profile
  $\pi' = \pi,u^{l+p},v^{k+p}$.  Note that
  $\pi' = \pi, u^{l-k}, u^{k+p}, v^{k+p}$.

  Suppose first that $l = k$ (i.e., $F_{G,\pi}(u) = F_{G,\pi}(v)$) and
  that $u \in L(\pi)$. Since
  $u \in L(u^{k+p},v^{k+p}) = L(u,v) = \{u,v\}$, we have
  $L(\pi') = L(\pi) \cap \{u,v\}$.
  By Lemma~\ref{lem-Lpiprime},
  $u,v \in L(\pi')$ and thus $u,v \in L(\pi)$.

  Suppose now that $l > k$ (i.e., $F_{G,\pi}(u) > F_{G,\pi}(v)$) and
  that $u \in L(\pi)$. Since $u \in L(u^{k+p},v^{k+p}) = \{u,v\}$ and
  $u \in L(u^{l-k}) = L(u) = \{u\}$, we have
  $L(\pi') = L(\pi) \cap \{u,v\} \cap \{u\} = \{u\}$, contradicting
  Lemma~\ref{lem-Lpiprime}.
\end{proof}

\begin{theorem}\label{ABCT-gcm}
  Any graph $G$ with connected medians is an ABCT-graph and
  an ABCT$^-$-graph.
\end{theorem}

\begin{proof}
  Since graphs with connected medians satisfy the triangle condition,
  Lemma~\ref{lem-L-wconv} holds for $G$. Consider an ABCT-function $L$
  on $V$ and a profile $\pi$. We first show that
  $L(\pi) \subseteq \Med_G(\pi)$. Pick $u \in L(\pi)$. By
  Theorem~\ref{th-cmed-p}(\ref{th-cmedp-0}) for $p=1$, if
  $u \notin \Med_G(\pi)$, then there exists $v \in N(u)$ such that
  $F_{G,\pi}(v) < F_{G,\pi}(u)$. Since By Lemma~\ref{lem-L-wconv}(2),
  $u \notin L(\pi)$, a contradiction.

  We now show that $\Med_G(\pi) \subseteq L(\pi)$. By
  Theorem~\ref{th-cmed-p}(\ref{th-cmedp-6}) for $p=1$, $\Med_G(\pi)$
  is connected in $G$. Consequently, if
  $L(\pi) \subsetneq \Med_G(\pi)$, there exists $u \in L(\pi)$ and
  $v \in \Med_G(\pi) \setminus L(\pi)$ such that $u \sim v$. Since
  $u, v \in \Med_G(\pi)$, $F_{G,\pi}(u) = F_{G,\pi}(v)$ and we obtain
  a contradiction with Lemma~\ref{lem-L-wconv}(1). Thus, the graphs
  with connected medians verifies the ABCT-property.  Since they
  satisfy the triangle condition, by Lemma~\ref{lem:TC-T-E1}, they
  also satisfy the ABCT$^-$-property.
\end{proof}

\begin{remark}
  Since the complete graph $K_3$ has two
  ABC-functions~\cite{McMoMuNoPo} and by Theorem~\ref{ABCT-gcm} only
  one ABCT-function (or ABCT$^-$-function), axioms (T) and (T$^-$) are
  independent from the axioms (A), (B), and (C).
\end{remark}

\subsection{Modular graphs with $G^2$-connected medians are
  ABC-graphs}\label{part:ABC-mod}

In this subsection we show that all modular graphs $G$ with
$G^2$-connected medians are ABC-graphs. This result generalizes the
results of~\cite{MuNo} and~\cite{McMoMuNoPo} stating that median
graphs and graphs satisfying the pairing property are
ABC-graphs. Indeed, in Section~\ref{sec:pairing}, we establish that
graphs with the pairing property are bipartite Helly graphs, and thus
they are modular with $G^2$-connected medians.

For a bipartite graph $G=(V,E)$ and a 2-pair $(u,v)$ of $G$, let
$X_{uv} = \{x \in V: d(u,x) < d(v,x)\}$,
$X_{vu} = \{x \in V: d(v,x) < d(u,x)\}$, and
$X_{uv}^{=} = \{x \in V: d(u,x) = d(v,x)\}$. Since $G$ is
bipartite, for any $x \in X_{uv}$, $d(v,x) = d(u,x)+2$ and
$u \in I(v,x)$. The following observation is trivial:

\begin{lemma}\label{obs-Lvx-bip}
  For any bipartite graph $G$, any 2-pair $(u,v)$ of $G$, any
  $x \in X_{uv}$ and $y \in X_{vu}$, for any ABC-function $L$, we have
  $u,v \in L(v,x) = I(v,x)$ and $u,v \in L(u,y) = I(u,y)$.
\end{lemma}

The following lemma is similar to Theorem~5 of~\cite{McMoMuNoPo}. The
difference is that here we consider a subgraph $K_{2,n}$ in an
arbitrary bipartite graph $G$.

\begin{lemma}\label{lem-K23}
  For any bipartite graph $G$, any 2-pair $(u,v)$ and any distinct
  vertices $w_1,w_2, \ldots, w_n\in I^{\circ}(u,v)$, with $n \ge 2$,
  for any ABC-function $L$, we have $u,v \in L(w_1,w_2,\ldots,w_n)$
  and $u,v \in L(u,v,w_1,w_2,\ldots,w_n)$.
\end{lemma}

\begin{proof}
  Let $\pi = (w_1,w_2,\ldots,w_n)$.  For any $1 \leq i \leq n$, we
  have $u,v \in L(w_i,w_{i+1}) = I(w_i,w_{i+1})$ (with the convention
  that $w_{n+1}=w_1$). Consequently,
  $u,v \in L(w_1,w_2)\cap L(w_2,w_3) \cap \cdots \cap L(w_{n-1},w_n)
  \cap L(w_n,w_1) = L(w_1,w_2,w_2,w_3,\ldots,w_{n-1},w_n,w_n,w_1) =
  L(\pi\pi) = L(\pi)$. Since $u,v \in L(u,v) = I(u,v)$,
  $u,v \in L(u,v)\cap L(\pi) = L(u,v,w_1,w_2,\ldots,w_n)$.
\end{proof}

\begin{lemma}\label{lem-3neighbors}
  For a modular graph $G$ and a 2-pair $(u,v)$, if there exists a
  profile $\pi$ such that $F_{\pi}$ is not pseudopeakless on $(u,v)$,
  then $I^{\circ}(u,v)$ contains at least three vertices
  $w_1,w_2,w_3$.
\end{lemma}

\begin{proof} Suppose that $I^{\circ}(u,v)$ contains at most two
  vertices $w_1,w_2$ (where $w_1$ and $w_2$ may coincide). Since $G$
  is modular, for any vertex $x$ of $\pi$ either $w_1$ or $w_2$
  belongs to $I(x,u)\cap I(x,v)\cap I(u,v)$, yielding
  $d(x,w_1)+d(x,w_2)\le d(x,u)+d(x,v)$.  Therefore,
  $F_{\pi}(w_1)+F_{\pi}(w_2)\le F_{\pi}(u)+F_{\pi}(v)$, contrary to
  the assumption that $F_{\pi}$ is not pseudopeakless on $(u,v)$.
\end{proof}

Therefore, we investigate now the properties of an ABC-function $L$ on
2-pairs $(u,v)$ of $G$ for which $I^{\circ}(u,v)$ contains at least
three vertices.

\begin{lemma}\label{lem-K33+}
  Let $G$ be a bipartite graph, $(u,v)$ be a 2-pair with
  $w_1,w_2,w_3\in I^{\circ}(u,v)$, and $z\in X_{uv}^{=}$.  Then
  $u,v \in L(u,v,w_1,w_2,w_3,z)$.
\end{lemma}

\begin{proof}
  Let $k = d(u,z) = d(v,z)$. We first consider the case where $u,v$
  belong to one of the intervals $I(w_1,z), I(w_2,z), I(w_3,z)$. Note
  that this covers the case when $z \in \{w_1,w_2,w_3\}$.

  \begin{claim}\label{claim-K33-uv}
    If $\max\{ d(z,w_1),d(z,w_2),d(z,w_3)\}=k+1$, then
    $u,v \in L(u,v,w_1,w_2,w_3,z)$.
  \end{claim}

  \begin{proof}
    Let $d(z,w_1)=k+1$.  Since $u,v \in L(w_1,z) = I(w_1,z)$,
    $u,v \in L(u,v) = I(u,v)$, and $u,v \in L(w_2,w_3) = I(w_2,w_3)$,
    necessarily
    $u,v \in L(w_1,z) \cap L(u,v) \cap L(w_2,w_3) =
    L(u,v,w_1,w_2,w_3,z)$. This ends the proof of the claim.
  \end{proof}

  By Claim~\ref{claim-K33-uv} and since $G$ is bipartite, we can
  assume that $d(w_1,z) = d(w_2,z) = d(w_3,z) = k-1$.

  \begin{claim}\label{claim-K33-toutourien}
    Either $u,v,w_1,w_2,w_3 \in L(u,v,w_1,w_2,w_3,z)$ or
    $u, v, w_1, w_2, w_3 \notin
    L(u,v,w_1,w_2,w_3,z)$.
  \end{claim}

  \begin{proof}
    By symmetry, it suffices to show that $u \in L(u,v,w_1,w_2,w_3,z)$
    iff $w_1 \in L(u,v,w_1,w_2,w_3,z)$.  Suppose that
    $L(u,v,w_1,w_2,w_3,z) \cap \{u,w_1\} \neq \varnothing$ and
    consider the profile $\pi' = (u,v,w_1,w_2,w_3,z,u,w_1)$. Note that
    $L(u,w_1) = I(u,w_1) = \{u,w_1\}$ and thus
    $L(\pi') = L(u,v,w_1,w_2,w_3,z) \cap L(u,w_1) =
    L(u,v,w_1,w_2,w_3,z) \cap \{u,w_1\}$. Note also that
    $u,w_1 \in L(u,z) = I(u,z)$, $u,w_1 \in L(u,v) = I(u,v)$,
    $u,w_1 \in L(w_1,w_2) = I(w_1,w_2)$, and
    $u,w_1 \in L(w_1,w_3) = I(w_1,w_3)$. Consequently,
    $u,w_1 \in L(u,z) \cap L(u,v) \cap L(w_1,w_2) \cap L(w_1,w_3) =
    L(u,z,u,v,w_1,w_2,w_1,w_3) = L(\pi') = L(u,v,w_1,w_2,w_3,z) \cap
    \{u,w_1\}$. Therefore, $u,w_1 \in L(u,v,w_1,w_2,w_3,z)$ and we are
    done. This ends the proof of the claim.
  \end{proof}

  Now, we prove the assertion of the lemma. Let
  $z' \in L(u,v,w_1,w_2,w_3,z)$. By Lemma~\ref{lem-xinL},
  $z' \in L(u,v,w_1,w_2,w_3,z')$ and
  $L(u,v,w_1,w_2,w_3,z') \subseteq L(u,v,w_1,w_2,w_3,z)$.  Suppose
  first that $d(u,z') \neq d(v,z')$ and assume without loss of
  generality that $d(u,z') < d(v,z')$. In this case,
  $d(v,z') = d(w_1,z')+1 = d(w_2,z')+1 = d(w_2,z')+1 =d(u,z')+2$.
  Then $u \in L(z',v) = I(v,z')$, $u \in L(w_2,w_3) = I(w_2,w_3)$, and
  $u \in L(u,w_1) = \{u,w_1\}$. Consequently,
  $L(u,v,w_1,w_2,w_3,z') = I(v,z') \cap I(w_2,w_3) \cap \{u,w_1\} =
  \{u\}$. Therefore
  $u \in L(u,v,w_1,w_2,w_3,z') \subseteq L(u,v,w_1,w_2,w_3,z)$ and we
  are done by Claim~\ref{claim-K33-toutourien}.

  Suppose now that $d(u,z') = d(v,z')$ and let $k' = d(u,z')$.  If
  $\max\{ d(z',w_1),d(z',w_2),d(z',w_3)=k'+1$, say $d(w_1,z') = k'+1$,
  by Claim~\ref{claim-K33-uv} applied to the profile
  $(u,v,w_1,w_2,w_3,z')$, we have
  $u,v \in L(u,v,w_1,w_2,w_3,z') \subseteq L(u,v,w_1,w_2,w_3,z)$ and
  we are done. Thus, we can assume that
  $d(w_1,z') = d(w_2,z') = d(w_3,z') = k'-1$. Consider the profile
  $\pi'' = (u,v,w_1,w_2,w_3,z',w_1,z')$. Since
  $z' \in L(u,v,w_1,w_2,w_3,z')$ and $z' \in L(w_1,z') = I(w_1,z')$,
  we have $L(\pi'') = L(u,v,w_1,w_2,w_3,z') \cap L(w_1,z')$. Observe
  that $w_1 \in L(u,z') = I(u,z')$, $w_1 \in L(v,z') = I(v,z')$,
  $w_1 \in L(w_1,w_2) = I(w_1,w_2)$, and
  $w_1 \in L(w_1,w_3) = I(w_1,w_3)$. Consequently,
  $w_1 \in L(u,z')\cap L(v,z')\cap L(w_1,w_2) \cap L(w_1,w_3) =
  L(u,z',v,z',w_1,w_2,w_1,w_3) = L(\pi'') = L(u,v,w_1,w_2,w_3,z') \cap
  L(w_1,z')$. Therefore,
  $w_1 \in L(u,v,w_1,w_2,w_3,z') \subseteq L(u,v,w_1,w_2,w_3,z)$ and
  we are done by Claim~\ref{claim-K33-toutourien}. This ends the proof
  of the lemma.
\end{proof}

For any profile $\pi$ on a bipartite graph $G$ and any 2-pair $(u,v)$,
let $\pi_{uv}$ and $\pi_{vu}$ be the restrictions of $\pi$ on the sets
$X_{uv}$ and $X_{vu}$, respectively, and $\pi_{uv}^{=}$ be the
restriction of $\pi$ on $X_{uv}^{=}$.

\begin{lemma}\label{lem-bip-Lpiprime}
  Let $G$ be a bipartite graph, $(u,v)$ be a 2-pair with
  $w_1,w_2,w_3\in I^{\circ}(u,v)$.  Then for any profile $\pi$, we
  have $u,v \in L(\pi')$ where
  $\pi' = \pi, u^{\ell+p}, v^{k+p}, w_1^p, w_2^p, w_3^p$ with
  $\ell = |\pi_{vu}|$, $k = |\pi_{uv}|$, and $p=|\pi_{uv}^{=}|$.
\end{lemma}

\begin{proof}
  Set $\pi_{uv} = (x_1,\ldots,x_k)$, $\pi_{vu} = (y_1,\ldots,y_{\ell})$,
  and $\pi_{uv}^{=} = (z_1,\ldots,z_p)$. Consider the profile
  $\pi' = \pi, u^{l+p}, v^{k+p}, w_1^p, w_2^p, w_3^p$.  By
  Lemma~\ref{obs-Lvx-bip},
  $u,v \in \bigcap_{i=1}^k L(v,x_i) = \bigcap_{i=1}^k I(v,x_i)$ and
  $u,v \in \bigcap_{i=1}^{\ell} L(u,y_i) = \bigcap_{i=1}^{\ell}
  I(u,y_i)$. Note also that
  $u,v \in \bigcap_{i=1}^p L(u,v,w_1,w_2,w_3,z_i)$ by
  Lemma~\ref{lem-K33+}.  Consequently,
  $u,v \in L(v^k,x_1, \ldots,x_k, u^{\ell},y_1, \ldots, y_{\ell}, u^p,
  v^p, w_1^p, w_2^p, w_3^p, z_1, \ldots, z_p) = L(x_1, \ldots, x_k,
  y_1, \ldots, y_{\ell}, z_1, \ldots, z_p, u^{{\ell}+p}, v^{k+p},
  w_1^p, w_2^p, w_3^p) = L(\pi')$.
\end{proof}

\begin{lemma}\label{lem-bip-L-wconv}
  Let $G$ be a bipartite graph and $(u,v)$ be a 2-pair with
  $w_1,w_2,w_3\in I^{\circ}(u,v)$.  For any profile $\pi$ on $G$ the
  following holds:
\begin{enumerate}[(1)]
  \item If $F_{G,\pi}(u) = F_{G,\pi}(v)$, then $u \in L(\pi)$ iff
    $v \in L(\pi)$.
  \item If $F_{G,\pi}(u) > F_{G,\pi}(v)$, then $u \notin L(\pi)$.
  \end{enumerate}
\end{lemma}

\begin{proof}
  Let $\ell = |\pi_{vu}|$, $k = |\pi_{uv}|$, and $p=|\pi_{uv}^{=}|$.
  Suppose without loss of generality that
  $F_{G,\pi}(u) \geq F_{G,\pi}(v)$. Then
  $F_{G,\pi}(u) - F_{G,\pi}(v) = |\pi_{vu}| - |\pi_{uv}| = \ell - k
  \geq 0$.  As in Lemma~\ref{lem-bip-Lpiprime}, we consider the
  profile $\pi' = \pi,u^{\ell+p},v^{k+p},w_1^p,w_2^p,w_3^p$.  Note
  that
  $\pi' = \pi, u^{\ell-k}, u^{k}, v^{k},u^p,v^p,w_1^p,w_2^p,w_3^p$.

  Suppose first that $F_{G,\pi}(u) = F_{G,\pi}(v)$ (i.e., $\ell = k$)
  and that $u \in L(\pi)$. Since
  $u \in L(u^{k},v^{k}) = L(u,v) = I(u,v)$ and
  $u \in L(u^p,v^p,w_1^p,w_2^p,w_3^p) = L(u,v,w_1,w_2,w_3)$ by
  Lemma~\ref{lem-K23}, we have
  $L(\pi') = L(\pi) \cap L(u,v,w_1,w_2,w_3) \cap I(u,v)$. By
  Lemma~\ref{lem-bip-Lpiprime}, $u,v \in L(\pi')$ and thus
  $u,v \in L(\pi)$.

  Suppose now that $F_{G,\pi}(u) > F_{G,\pi}(v)$ (i.e., $\ell > k$)
  and that $u \in L(\pi)$. Since
  $u \in L(u^{k},v^{k}) \cap L(u^p,v^p,w_1^p,w_2^p,w_3^p) = I(u,v)
  \cap L(u,v,w_1,w_2,w_3)$ and $u \in L(u^{\ell-k}) = L(u) = \{u\}$,
  we have
  $L(\pi') = L(\pi) \cap I(u,v) \cap L(u,v,w_1,w_2,w_3) \cap \{u\} =
  \{u\}$, contradicting Lemma~\ref{lem-bip-Lpiprime}.
\end{proof}

\begin{theorem}\label{th-G2conn}
  Any modular graph $G=(V,E)$ with $G^2$-connected medians is an
  ABC-graph.  In particular, any bipartite Helly graph is an
  ABC-graph.
\end{theorem}

\begin{proof}
  Consider an ABC-function $L$ on $V$ and a profile $\pi\in V^*$. We
  first show that $L(\pi) \subseteq \Med_G(\pi)$. Pick any
  $u \in L(\pi)$.  If $u\notin \Med_G(\pi)$, then by
  Theorem~\ref{th-cmed-p}(\ref{th-cmedp-0}) for $p=2$ there exists a
  vertex $v$ such that $1\le d(u,v)\le 2$ and
  $F_{G,\pi}(v) < F_{G,\pi}(u)$. Since $G$ is bipartite, $G$ satisfies
  the triangle condition.  Therefore, if $d(u,v)=1$ we obtain a
  contradiction with Lemma~\ref{lem-L-wconv}(2).  Thus, we can suppose
  that $d(u,v)=2$ and that $F_{G,\pi}(w)\ge F_{G,\pi}(u)$ for any
  neighbor $w$ of $u$. By Lemma~\ref{lem-3neighbors} applied to the
  2-pair $(u,v)$, there exists three distinct vertices
  $w_1,w_2,w_3 \in I^{\circ}(u,v)$.  Consequently, by
  Lemma~\ref{lem-bip-L-wconv}(2), $u \notin L(\pi)$, a
  contradiction. This shows that $L(\pi) \subseteq \Med_G(\pi)$.

  Now we prove the converse inclusion $\Med_G(\pi) \subseteq L(\pi)$.
  Suppose by way of contradiction that $L(\pi)$ is a proper subset of
  $\Med_G(\pi)$.  Pick two vertices $u \in L(\pi)$ and
  $v \in \Med_G(\pi) \setminus L(\pi)$ minimizing $d(u,v)$.  Since
  $\Med_G(\pi)$ is $G^2$-connected, $d(u,v) \leq 2$.  Since
  $u, v \in \Med_G(\pi)$, $F_{G,\pi}(u) = F_{G,\pi}(v)$. If
  $d(u,v) = 1$, we obtain a contradiction with
  Lemma~\ref{lem-L-wconv}(1). If $d(u,v)=2$, by our choice of $u$ and
  $v$, we must have $F_{G,\pi}(w)>F_{G,\pi}(u)$ for any
  $w\in I^{\circ}(u,v)$. Thus, by Lemma~\ref{lem-3neighbors} there
  exists three distinct vertices $w_1,w_2,w_3 \in I^{\circ}(u,v)$ and
  we obtain a contradiction with Lemma~\ref{lem-bip-L-wconv}.

  The second assertion is a consequence of~\cite[Proposition
  62]{GpConMed} establishing that bipartite Helly graphs have $G^2$
  connected medians since bipartite Helly graphs are modular.
\end{proof}

\section{Graphs with the pairing property}\label{sec:pairing}

In this section, we investigate the class of graphs with the pairing
property and the (larger) class of graphs with the double-pairing
property. McMorris et al.~\cite{McMoMuNoPo} established that the
graphs with the pairing property are ABC-graphs.  We show that this
also holds for graphs with the double-pairing property.  We also prove
that both those classes of graphs are proper subclasses of bipartite
Helly graphs, showing that Theorem~\ref{th-G2conn} is a strict
generalization of~\cite[Theorem~4]{McMoMuNoPo}.  We characterize the
graphs with the pairing property in a local-to-global way as bipartite
Helly graphs where each ball of radius 2 satisfies the pairing
property. Finally, we show that the problem of deciding if a graph
satisfies the double-pairing property is in co-NP.

\subsection{Pairing and double-pairing properties}

For an even profile $\pi$ of length $k=2n$ of $G$, a \emph{pairing}
$P$ of $\pi$ is a partition of $\pi$ into $n$ disjoint pairs. For a
pairing $P$, define $D_{\pi}(P)=\sum_{\{ a,b\}\in P} d(a,b)$.  A
pairing $P$ of $\pi$ maximizing the function $D_{\pi}$ is called a
\emph{maximum pairing}.  The notion of pairing was defined by Gerstel
and Zaks~\cite{GeZa}; they also proved the following weak duality
between the functions $F_{\pi}$ and $D_{\pi}$:

\begin{lemma}\label{pairing1}
  For any even profile $\pi$ of length $k=2n$ of $G$, for any pairing
  $P$ of $\pi$, and for any vertex $v$ of $G$, $D_{\pi}(P)\le F_{\pi}(v)$
  and the equality holds if and only if
  $v \in \bigcap_{\{a,b\} \in P} I(a,b)$.
\end{lemma}

\begin{proof}
  Let $P=\left\{ \{ a_i,b_i\}: i=1,\ldots,n\right\}$. By the triangle
  inequality,
  $D_{\pi}(P)=\sum_{i=1}^n d(a_i,b_i)\le \sum_{i=1}^n
  (d(a_i,v)+d(v,b_i)) = \sum_{x\in \pi} d(v,x)=F_{\pi}(v)$. Observe
  that $D_{\pi}(P) = F_{\pi}(v)$ if and only if
  $d(a_i,b_i) = d(a_i,v)+d(v,b_i)$ for all $1 \leq i \leq n$, i.e., if
  and only if $v \in I(a_i,b_i)$ for all $1 \leq i \leq n$.
\end{proof}

We say that a graph $G$ satisfies the \emph{pairing property} if for
any even profile $\pi$ there exists a pairing $P$ of $\pi$ and a
vertex $v$ of $G$ such that $D_{\pi}(P)=F_{\pi}(v)$, i.e., the
functions $F_{\pi}$ and $D_{\pi}$ satisfy the strong duality. Such a
pairing is called a \emph{perfect pairing}. We say that a graph
$G$ satisfies the \emph{double-pairing property} if for any profile
$\pi$, the profile $\pi^{2}$ admits a perfect pairing. Clearly a graph
satisfying the pairing property also satisfies the double-pairing
property.

By Lemma~\ref{pairing1}, the pairing property of~\cite{GeZa} coincides
with the \emph{intersecting-intervals property}
of~\cite{McMoMuNoPo}. By~\cite[Theorem 4]{McMoMuNoPo} the graphs
satisfying the pairing property are ABC-graphs.  It was shown
in~\cite{GeZa} that trees satisfy the pairing property. More
generally, it was shown in~\cite{McMoMuRo} and rediscovered
in~\cite{ChFeGoVa} that cube-free median graphs also satisfy the
pairing property. In was proven in~\cite{McMoMuNoPo} that the complete
bipartite graph $K_{2,n}$ satisfies the pairing property. As observed
in~\cite{McMoMuNoPo}, the $3$-cube is a simple example of a graph not
satisfying the pairing property.  The investigation of the structure
of graphs with the pairing property was formulated as an open problem
in~\cite{GeZa}.

We use the following lemma, whose first assertion follows
from~\cite[Theorem 4]{McMoMuNoPo}. This Lemma~\ref{pairing2}
establishes that any graph with the pairing or the double-pairing
property is an ABC-graph.

\begin{lemma}\label{pairing2}
  Let $L$ be an ABC-function on a graph $G$. Then for any even profile
  $\pi$ in $G$ that admits a perfect pairing, we have
  $L(\pi) = \Med(\pi)$.  Furthermore, for any profile $\pi$ such that
  $\pi^{2k}$ has a perfect pairing for some $k \geq 1$, we have
  $L(\pi) = \Med(\pi)$.
\end{lemma}

\begin{proof}
  Consider a perfect pairing $P$ of an even profile $\pi$ and let
  $v \in V$ such that $D_\pi(P) = F_\pi(v)$.  By Lemma~\ref{pairing1},
  $v$ is a median of $\pi$ and $v \in \bigcap_{\{a,b\} \in P}
  I(a,b)$. Consequently, $\Med(\pi) = \bigcap_{\{a,b\} \in P} I(a,b)$.
  By (B) and (C),
  $L(\pi) = \bigcap_{\{a,b\} \in P} L(a,b) = \bigcap_{\{a,b\} \in P}
  I(a,b)$, and thus $L(\pi) = \Med(\pi)$.

  Consider now a profile $\pi$ and a perfect pairing $P$ of
  $\pi^{2k}$. By the first assertion, we have
  $L(\pi^{2k}) = \Med(\pi^{2k})$. By (C), we have
  $L(\pi^{2k}) = \bigcap_{i=1}^{2k }L(\pi) = L(\pi)$ and
  $\Med(\pi^{2k}) = \Med(\pi)$.
\end{proof}

Given a profile $\pi$ on a graph $G =(V,E)$, deciding whether $\pi$
(or $\pi^{2}$) admits a perfect pairing can be reduced to a perfect
$\pi$-matching problem as follows. Consider a vertex $u \in \Med(\pi)$
and the auxiliary graph $A_u=(V,E_u)$ where $vw \in E_u$ if and only
$u \in I(v,w)$ in $G$, i.e., $d(v,w)=d(v,u)+d(u,w)$. Then $\pi$ admits
a perfect pairing if and only if the graph $A_u$ has a perfect
$\pi$-matching, by which we mean there exists a multiset $P$ of edges of $E_u$
such that each vertex $v \in V$ belongs to exactly $\pi(v)$ edges of
$P$.

In fact, the previous remark holds for any weight function
$b: V \to \R^+$ (note that profiles correspond to the weight functions with
integer values).  Analogously to the definitions of $F_\pi$ and
$\Med_\pi$, for a weight function $b$, let
$F_b: v \in V \mapsto \sum_{z \in V} b(z) d(v,z)$ and let $\Med(b)$ be
the set of vertices minimizing $F_b$.  Recall that given a weight
function $b: V \to \R^+$, a fractional perfect $b$-matching of
$A_u=(V,E_u)$ is a weight function $x : E_u \to \R^+$ on the edges of
$A_u$ such that for each $v \in V$,
$\sum_{e \in E_u: v \in e} x(e) = b(v)$.

\begin{lemma}\label{lem-bMatchMed}
  For any weight function $b: V \to \R^+$ such that $A_u$ admits a
  fractional perfect $b$-matching, we have $u \in \Med(b)$.
\end{lemma}

\begin{proof}
  Consider a fractional perfect $b$-matching $x: E_u \to \R^+$ of
  $A_u$ and observe by triangle inequality that for any $v \in V$,
  \[ F_b(v) = \sum_{z \in V} b(z) d(v,z) = \sum_{z \in V} (\sum_{e:z
      \in e} x(e)) d(v,z) = \sum_{e=zz' \in E_u} x(e) (d(v,z)+d(v,z'))
    \geq \sum_{e=zz' \in E_u} x(e) d(z,z').\]

  The last inequality is an equality if and only if $v \in I(z,z')$
  for all edges $e = zz'$ such that $x(e) > 0$. By the definition of
  the edges of $A_u$, we thus conclude that
  $ F_b(u) = \sum_{e=zz' \in E_u} x(e) d(z,z')$ and consequently, that
  $u \in \Med(b)$.
\end{proof}

Since deciding if a graph $G$ admits a fractional perfect $b$-matching
can be done in (strongly) polynomial
time~\cite{Anstee,Marsh,Pulleyblank}, deciding if a profile $\pi$ on a
graph $G$ has a perfect pairing can be done in polynomial time.

For a given weight function $b$, the set of all fractional perfect
$b$-matchings of $A_u$ is described by the following polytope
$\dM(b,u)$ (see~\cite[Chapter 31]{SchrijverA}):
\[
  \dM({b,u}) \quad
  \begin{cases}
    \sum_{e:v \in e} x(e) = b(v) &  \text{for all } v \in V\\
    x(e) \geq 0 & \text{for all } e \in E_u\\
  \end{cases}
\]
If $b$ is a profile $\pi$, i.e., $b(v) \in \N$ for all $v \in V$, then
the polytope $\dM({b,u})$ is either empty (and $A_u$ does not have any
fractional perfect $b$-matching) or each vertex
$\bx={(x(e))}_{e\in E_u}$ of $\dM(b,u)$ is
half-integral~\cite{Balinski} (see also~\cite[Theorem
30.2]{SchrijverA}), i.e., for each $e \in E_u$, we have
$x(e) \in \{0,\frac{1}{2},1\}$.
Consequently, we have the following result.

\begin{lemma}\label{lem-integral-pm}
  Given a graph $G$, an integral weight function $b: V \to \N$ such
  that $b(v)$ is even for all $v \in V$, and a vertex $u \in V$, then
  either the auxiliary graph $A_u$ has an integral perfect
  $b$-matching, or $A_u$ has no fractional perfect $b$-matching.
\end{lemma}

Lemma~\ref{pairing2} suggests that graphs with the $2k$-pairing
property (i.e., graphs such that for any profile $\pi$, the profile
$\pi^{2k}$ has a perfect pairing) can be interesting generalizations
of graphs with the pairing or double-pairing properties. However, the
following result shows that they coincide with the graphs with the
double-pairing property.

\begin{lemma}\label{lem-integral-k}
  Given a graph $G=(V,E)$, a vertex $u \in V$, and a profile $\pi$ on
  $G$, if $\dM(\pi,u)$ is non-empty, then $\pi^2$ admits a perfect
  pairing.  Consequently, if $G$ has a perfect $\pi^{2k}$-pairing for
  some integer $k \geq 1$, then $G$ has a perfect $\pi^2$-pairing.
\end{lemma}

\begin{proof}
Observe that if $\dM(\pi,u)$ is non-empty, then $\dM(\pi^2,u)$ is
  also non-empty. By Lemma~\ref{lem-integral-pm}, $A_u$ admits an
  integral perfect $\pi^2$-matching. Consequently, $\pi^2$ admits a
  perfect pairing.

  To prove the second assertion, consider a profile $\pi$ on $G=(V,E)$
  such that $\pi^{2k}$ admits a perfect pairing and pick any
  $u \in \Med(\pi^{2k}) = \Med(\pi)$.  Observe that if $\pi^{2k}$ has
  a perfect pairing, then $A_u$ has a fractional perfect
  $\pi$-matching (where $x(e)$ is a multiple of $\frac{1}{2k}$ for
  each $e \in E_u$). Consequently, the polytope $\dM({\pi,u})$ is
  non-empty and thus $\pi^2$ admits a perfect pairing.
\end{proof}

\subsection{(Double-)pairing property implies bipartite Hellyness}

We now establish that graphs with the double-pairing property (and
thus with the pairing property) are bipartite Helly graphs. Recall
that bipartite Helly graphs are graphs in which the collection of
half-balls of $G$ has the Helly property.
The bipartite Helly graphs have been characterized in several nice
ways by Bandelt, D\"ahlmann, and Sch\"utte~\cite{BaDaSch} (the
papers~\cite{Ba_C,BaDaSch} also characterize the bipartite Helly
graphs via medians, Condorcet vertices, and plurality vertices):

\begin{theorem}[\cite{BaDaSch}]\label{absolute-retracts}
  For a bipartite graph $G$, the following conditions are equivalent:
  \begin{enumerate}[(1)]
  \item\label{th-bar-1} $G$ is a bipartite Helly graph;
  \item\label{th-bar-2} $G$ is a modular graph of breadth at most two;
  \item\label{th-bar-3} $G$ is a modular graph such that every induced
    $B_n$ ($n\ge 4$) extends to $\widehat{B}_n$ in $G$;
  \item\label{th-bar-4} $G$ satisfies the following interval
    condition: for any vertices $u$ and $v$ with $d(u,v)\ge 3$, the
    neighbors of $v$ in $I(u,v)$ have a second common neighbor $x$ in
    $I(u,v)$.
  \end{enumerate}
\end{theorem}

The graph $B_n$ is the bipartite complete graph $K_{n,n}$ minus a
perfect matching, i.e., the bipartition of $B_n$ is defined by
$a_1,\ldots,a_n$ and $b_1,\ldots,b_n$ and $a_i$ is adjacent to $b_j$
if and only if $i\ne j$. The graph $\widehat{B}_n$ is obtained from
$B_n$ by adding two adjacent new vertices $a$ and $b$ such that $a$ is
adjacent to all vertices $b_1,\ldots,b_n$ and $b$ is adjacent to all
vertices $a_1,\ldots,a_n$.  Recall that a graph $G$ is of
\emph{breadth at most two}~\cite{BaDaSch} if for any vertex $u$ and
any set of vertices $W$ of $G$, the following implication holds:
$\bigcap_{v\in W} I(u,v)=\{ u\}$ implies that there exists two
vertices $w',w''\in W$ such that $I(u,w')\cap I(u,w'')=\{ u\}$.

\begin{proposition}\label{pairing->bar}
  If a graph $G$ satisfies the double-pairing property, then $G$ is a
  bipartite Helly graph.
\end{proposition}

\begin{proof}
  First we prove that $G$ is bipartite. Suppose the contrary and let
  $C$ be an odd cycle of minimum length of $G$. Then $C$ is an
  isometric cycle of $G$.  Let $uv$ be an edge of $C$ and $w$ be the
  vertex of $C$ opposite to $uv$. Let $d(u,w)=d(v,w)=k\ge 1$. Consider
  the profile $\tau = (u,v,w)$ and let $\pi =
  \tau^2=(u,u,v,v,w,w)$. Then $F_{\pi}(u)=F_{\pi}(v)=2k+2$.
For any other vertex $x$ of $G$, let $k_1=d(x,u) \geq 1$,
  $k_2=d(x,v) \geq 1$, and $k_3=d(x,w)$, where $k_1+k_3\ge k$ and
  $k_2+k_3\ge k$.
Notice that $F_{\pi}(x)=2k_1+2k_2+2k_3\ge 2k+2k_2\ge 2k+2$.  Thus
  $u,v\in \Med(\pi)$. Up to symmetry, $\pi$ has three non-trivial
  pairings: $P_1=\{ \{ u,v\}, \{ v,w\}, \{ w,u\}\}$,
  $P_2=\{ \{u,w\}, \{ u,w\}, \{ v,v\}\}$,
  $P_3=\{ \{u,v\}, \{ u,v\}, \{ w,w\}\}$. Notice that
  $D_{\pi}(P_1)=2k+1, D_{\pi}(P_2)=2k$, and $D_{\pi}(P_3)=2$, i.e.,
  $D_{\pi}(P_i)<F_{\pi}(u)$ for $i=1,2,3$, and thus $\pi = \tau^2$
  does not admit a perfect pairing.

  By Theorem~\ref{absolute-retracts}, $G$ is a bipartite Helly graph
  if and only if $G$ is a modular graph and every induced $B_n$
  ($n\ge 4$) extends to $\widehat{B}_n$. To prove that $G$ is modular,
  since $G$ is bipartite, it suffices to show that $G$ satisfies the
  quadrangle condition. Suppose by way of contradiction that $G$
  contains four vertices $u,v,w,y$ such that $y\sim u,v$ and
  $d(u,w)=d(v,w)=d(y,w)-1=k\ge 2$ and $I(u,w)\cap I(v,w)=\{ w\}$.
  Consider the profile $\tau = (u,v,w)$ and let
  $\pi = \tau^2=(u,u,v,v,w,w)$. Then
  $F_{\pi}(u)=F_{\pi}(v)=2k+4$. Pick any other vertex $x$ of $G$ and
  let $d(x,u)=k_1\ge 1, d(x,v)=k_2\ge 1$, and $d(x,w)=k_3\ge 0$.  Then
  $k_1+k_3\ge k$ and $k_2+k_3\ge k$.
Notice that $F_{\pi}(x)=2k_1+2k_2+2k_3$. Suppose that
  $F_{\pi}(x)<F_{\pi}(u)$, i.e., $2k_1+2k_2+2k_3<2k+4$. Since
  $k_1+k_3\ge k$ (resp. $k_2+k_3 \geq k$) this is possible only if
  $k_2=1$ (resp. $k_1=1$).
Thus $x$ is a common neighbor of $u$ and $v$.  Since
  $I(u,w)\cap I(v,w)=\{ w\}$ and $G$ is bipartite, we must have
  $k_3=k+1$, i.e., $F_{\pi}(x)=2k+6>2k+4$, a contradiction. This shows
  that $u,v \in \Med(\pi)$. Again, up to symmetry, $\pi$ has three
  non-trivial pairings:
  $P_1=\{ \{ u,v\}, \{ v,w\}, \{ w,u\}\}, P_2=\{ \{u,w\}, \{ u,w\}, \{
  v,v\}\}$, and $P_3=\{ \{u,v\}, \{ u,v\}, \{ w,w\}\}$. Note that
  $D_{\pi}(P_1)=2k+2, D_{\pi}(P_2)=2k$, and $D_{\pi}(P_3)=4$, i.e.,
  $D_{\pi}(P_i)<F_{\pi}(u)$ for $i=1,2,3$, and thus $\pi = \tau^2$
  does not admit a perfect pairing.  Consequently, any graph $G$
  satisfying the double-pairing property is modular.

  It remains to prove that every induced $B_n$
  ($n\ge 4$) extends to $\widehat{B}_n$. Suppose by way of
  contradiction that $G$ contains induced $B_n, n\ge 3$ which are not
  included in $\widehat{B}_n$ and pick such a $B_n$ with a maximum
  number of vertices (this maximum exists since $G$ is
  finite). Suppose that the vertex set of $B_n$ is defined by the
  bipartition
  $\left\{a_1,\ldots,a_n\right\}\cup \left\{
    b_1,\ldots,b_n\right\}$. Consider the profile
  $\tau= \left\{ a_1,\ldots,a_n\right\} $ and let
  $\pi=\tau^2=\left\{ a_1,a_1,\ldots,a_n,a_n\right\}$. One can easily
  see that if $P$ is a maximum pairing of $\pi$, then $D_{\pi}(P)=2n$.
  Notice that $F_{\pi}(b_i)=2(n-1)+6=2n+4$ and
  $F_{\pi}(a_i)=4n-4$. Therefore, if $\pi$ has a perfect pairing, then
  there exists a vertex $x\ne a_1,\ldots,a_n,b_1,\ldots,b_n$ such that
  $F_{\pi}(x)=D_{\pi}(P)=2n$. If in the unique bipartition of the
  vertices of $G$, $x$ belongs to the same part as $a_1,\ldots,a_n$,
  then $d(x,a_i)\ge 2$ and thus $F_{\pi}(x)\ge 4n$. Therefore, $x$
  belongs to the same part as $b_1,\ldots,b_n$ and to have
  $F_{\pi}(x)=D_{\pi}(P)=2n$ necessarily $x$ must be adjacent to all
  vertices $a_1,\ldots,a_n$. Similarly, there exists a vertex
  $y\ne a_1,\ldots,a_n$ adjacent to all vertices $b_1,\ldots,b_n$. If
  the vertices $x$ and $y$ are not adjacent in $G$, then $x,y$
  together with $a_1,\ldots, a_n,b_1,\ldots,b_n$ define a $B_{n+1}$,
  contrary to the maximality choice of $B_n$.  Thus, any graph $G$
  satisfying the pairing property is a bipartite Helly graph.
\end{proof}

\subsection{Bipartite Helly graphs without the pairing property}\label{subs-bipHwoPP}
Unfortunately, there exist bipartite Helly graphs that do not satisfy
the double-pairing property, and thus that do not satisfy the pairing
property.  To construct them, we relate the bipartite Helly graphs
with the Helly hypergraphs.  A hypergraph $\cH=(X,\cE)$ consists of a
set $X$ and a family $\cE$ of subsets of $X$. A hypergraph $\cH$ is
called a \emph{Helly hypergraph}~\cite{Berge} if for any subfamily
$\cE'$ of $\cE$ of pairwise intersecting sets the intersection
$\bigcap \cE'$ is nonempty. For any vertex $u$ of $G$ we can define a
hypergraph ${\mathcal H}_u$ in the following way: the ground-set of
${\mathcal H}_u$ is the set $N(u)$ of neighbors of the vertex $u$ and
for each vertex $v$ of $G$, the set $R_v:=N(u)\cap I(u,v)$ is a
hyperedge of ${\mathcal H}_u$.  Then the condition~\ref{th-bar-2} of
Theorem~\ref{absolute-retracts} can be rephrased in the following way:

\begin{lemma}\label{Helly<->bar1}
  A graph $G$ is a bipartite Helly graph if and only if $G$ is modular
  and for any vertex $u$, ${\mathcal H}_u$ is a Helly hypergraph.
\end{lemma}

The next lemma allows us to construct bipartite Helly graphs from
Helly hypergraphs. With any hypergraph $\cH=(X,\cE)$ we associate the
following graph $R(\cH)$, which we call the \emph{incidence graph} of
$\cH$.  The vertex-set of $R(\cH)$ consists of the set $X$, a vertex
$u$ which is adjacent to all vertices of $X$, and a vertex $v_H$ for
each hyperedge $H\in \cE$ which is adjacent exactly to all vertices of
$X$ which belong to $H$. We call the vertices $x\in X$
$v$-\emph{vertices} and the vertices $v_H, H\in \cH$
$h$-\emph{vertices}.  Notice that $R(\cH)$ is a bipartite graph of
diameter at most 4.

\begin{lemma}\label{Helly<->bar2}
  A hypergraph $\cH=(X,\cE)$ is Helly if and only if its incidence
  graph $R({\mathcal H})$ is a bipartite Helly graph.
\end{lemma}

\begin{proof}
  One direction follows from Lemma~\ref{Helly<->bar1}.  Conversely,
  suppose that $\cH$ is Helly.  Notice that for any two hyperedges
  $H,H'$, $d(v_H,v_{H'})=4$ if $H\cap H'=\varnothing$ and
  $d(v_H,v_{H'})=2$ if $H\cap H'=\varnothing$.  Notice also that
  $d(x,x')=2$ for any $x,x'\in X$ and $d(u,v_H)=2, d(x,v_H)\le 3$ for
  any $H\in \cE$ and $x\in X$.  We assert that $R(\cH)$ satisfies the
  condition~\ref{th-bar-4} of Theorem~\ref{absolute-retracts}.  We
  have to consider two cases. First, consider two vertices
  $v_H,v_{H'}$ with $d(v_H,v_{H'})=4$. Then $H$ and $H'$ are disjoint
  and the neighbors of $v_H$ in $I(v_H,v_{H'})$ are exactly the
  vertices of $H$. They have $u$ as the second common neighbor in
  $I(v_H,v_{H'})$. Second, consider two vertices $x\in X$ and $v_H$
  with $d(x,v_H)=3$. This implies that $x\notin H$. Then again the
  neighbors of $v_H$ in $I(v_H,x)$ have $u$ as the second common
  neighbor in $I(v_H,x)$. On the other hand, the neighbors of $x$ in
  $I(x,v_H)$ are the vertex $u$ and all the vertices $v_{H'}$ such
  that $H'\cap H\ne\varnothing$.  Since $x\in H'\cap H''$ for any two
  such vertices $v_{H'},v_{H''}$, applying the Helly property to $H$
  and all hyperedges $H'$ such that $v_{H'}\in I(x,v_H)$ we find a
  common point $y\in X$.  In $R(\cH)$, $y$ is adjacent to $v_H$, all
  $v_{H'}\in I(x,v_H)$, and to $u$. This shows that $R(\cH)$ is a
  bipartite Helly graph.
\end{proof}

There is a simple way to construct Helly hypergraphs from graphs via
hypergraph duality. The \emph{dual} of a hypergraph $\cH=(X,\cE)$ is
the hypergraph $\cH^*=(X^*,\cE^*)$ whose vertex-set $X^*$ is in
bijection with the edge-set $\cE$ of $\cH$ and whose edge-set $\cE^*$
is in bijection with the vertex-set $X$, namely $\cE^*$ consists of
all $S_x=\left\{ H_j\in \cE: x\in H_j\right\}, x\in X$. By definition,
${(\cH^*)}^*=\cH$. The \emph{clique hypergraph} of a graph $B=(V,E)$
is the hypergraph ${\mathcal C}(B)$ whose vertices are the vertices of
$B$ and whose hyperedges are the maximal cliques of $B$. The following
is a standard fact from hypergraph theory:

\begin{lemma}[\cite{Berge}]\label{dual}
  The dual ${({\mathcal C}(B))}^*$ of the clique hypergraph of any graph
  $B$ is a Helly hypergraph.
\end{lemma}

\begin{figure}[h]\centering
  {\includegraphics[scale=0.8]{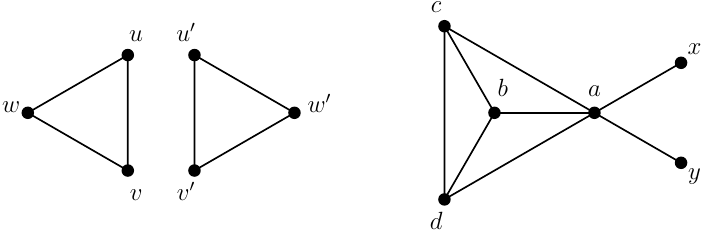}}
  \caption{The graphs from the proof of
    Proposition~\ref{bar->notpairing}}\label{fig-cex-PP}
\end{figure}

\begin{proposition}\label{bar->notpairing}
  There exist bipartite Helly graphs which do not satisfy the pairing
  property (respectively, the double-pairing property).
\end{proposition}

\begin{proof}
  We construct a bipartite Helly graph not satisfying the pairing
  property that has the form $R:=R({({\mathcal C}(B))}^*)$ for a
  specially chosen graph $B$. That $R$ is bipartite Helly follows from
  Lemmas~\ref{Helly<->bar2} and~\ref{dual}. By definition of the dual
  hypergraph, the vertices of ${({\mathcal C}(B))}^*$ are the maximal
  cliques of $B$ and the hyperedges of ${({\mathcal C}(B))}^*$ are in
  bijection with the vertices of $B$. For each vertex $b$ of $B$ we
  denote by $H(b)$ the corresponding hyperedge of
  ${({\mathcal C}(B))}^*$; $H(b)$ consists of all maximal cliques of $B$
  containing $b$. Therefore the $h$-vertices of $R$ are in bijection
  with the vertices of $B$: $v_{H(b)}\leftrightarrow b$. Notice also
  that for any distinct $b,b' \in V(B)$, $d(v_{H(b)},v_{H(b')}) = 2$
  if and only if $b \sim b'$.

  As $B$ we consider a graph with an even number $2m$ of vertices, we
  consider the even profile on $B$ of size $2m$
  $\pi=(v_{H(b)}: b\in V(B))$, and we consider the double-profile
  $\tau = \pi^2$.  We want to ensure that the central vertex $u$ of
  $R$ is in $\Med(\pi) = \Med(\tau)$. Observe that $F_{\pi}(u) =
  4m$. Moreover, for any $v$-vertex $x$ in $R$, $x$ corresponds to a
  maximum clique $K$ in $B$ and the set of neighbors of $x$ in $R$ is
  precisely $K \cup \{u\}$. Consequently,
  $F_{\pi}(x) = |K| + 3(2m-|K|)$ and thus $F_{\pi}(u) \leq F_{\pi}(x)$
  if and only if $|K| \leq m$.  Thus we require that all maximal
  cliques of $B$ have size at most $m$; equivalently, all maximal
  stable sets of $\overline{B}$ must have size at most $m$. For any
  $h$-vertex $v_{H(b)}$,
  $F_{\pi}(v_{H(b)}) = 2k + 4(2m-1-k) = 8m -2k -4$ where
  $k = |N_B(b)| \leq 2m-1$. Since we want to have
  $4m = F_{\pi}(u) \leq F_{\pi}(v_{H(b)}) = 8m-2k-4$, i.e.,
  $k \leq 2m-2$, we require that every vertex of $B$ is not adjacent
  to at least one vertex of $B$; equivalently, the minimum degree of
  $\overline{B}$ must be at least $1$.

  If the profile $\pi$ (respectively, $\tau$) of $R$ would satisfy the
  pairing property, then we can find a pairing $P$ of $\pi$
  (respectively, of $\tau$) such that for any pair
  $\left\{ v_{H(b)},v_{H(b')}\right\}\in P$ we must have
  $u\in I(v_{H(b)},v_{H(b')})$. As we noticed above, this is
  equivalent to require that $H(b)\cap H(b')=\varnothing$, which is
  obviously equivalent to the requirement that the vertices $b$ and
  $b'$ are not adjacent in $B$. Therefore, the profile $\pi$
  (respectively, $\tau$) of $R$ admits a perfect pairing if and only
  if the complement $\overline{B}$ of $B$ has a perfect matching
  (respectively, a $\tau$-perfect matching). Since $\tau = \pi^2$, the
  weight of each node of $B$ in $\tau$ is 2, and thus by
  Lemma~\ref{lem-integral-pm} $B$ admits a $\tau$-perfect matching if
  and only if $B$ admits a fractional perfect matching.

  Therefore, we have to construct a graph $C$ with $2m$ vertices with
  no isolated vertices in which the maximum stable set has size at
  most $m$ and that does not have a perfect matching (respectively, a
  fractional perfect matching). Then, we can take $B = \overline{C}$
  and $R({(\cC(B))}^*)$ will be a bipartite Helly graph that does not
  satisfy the pairing property (respectively, the double-pairing
  property). For the pairing property, we can take the disjoint union
  of two triangles as $C$ (see Figure~\ref{fig-cex-PP}, left). In this
  graph, there is no isolated vertex, any maximum stable set is of
  size 2, and clearly it does not contain a perfect matching.  For the
  double-pairing property, we can take as $C$ the graph on $6$
  vertices obtained by considering a complete graph on four vertices
  $a,b,c,d$ in which we added two nodes $x, y$ that are only adjacent
  to $a$ (see Figure~\ref{fig-cex-PP}, right). In this graph, there is
  no isolated vertex, any stable set is of size at most $3$, and it
  does not contain a fractional perfect matching. Indeed, the stable
  set $S=\{x,y\}$ is of size $2$, but $N(S) = \{a\}$ is of size $1$,
  preventing $C$ from having a fractional perfect matching.
\end{proof}

\subsection{A local-to-global characterization of graphs with the
  pairing property}
In this subsection, we provide a local-to-global characterization of
bipartite Helly graphs with the pairing property.
Note that balls in bipartite Helly graphs induce bipartite Helly
graphs:

\begin{lemma}\label{lem-boules-Helly}
  If $G$ is a bipartite Helly graph, then for every vertex $u$ and
  every integer $k$, the subgraph of $G$ induced by the ball $B_k(u)$
  is an isometric subgraph of $G$, which is a bipartite Helly graph.
\end{lemma}

\begin{proof}
  Let $H=G[B_k(u)]$. First notice that $H$ is an isometric subgraph of
  $G$. Indeed, pick any $x,y\in B_k(u)$. Since $G$ is modular, the
  triplet $x,y,u$ has a median vertex $z$.  Since $x,y\in B_k(u)$ and
  $z\in I(u,x)\cap I(u,y)$, the vertex $z$ and the intervals $I(x,z)$
  and $I(z,y)$ belong to the ball $B_k(u)$. Since $z\in I(x,y)$ and
  $I(x,z)\cup I(z,y)\subset B_k(u)$, $x$ and $y$ can be connected in
  $H$ by a shortest path of $G$, thus $H$ is an isometric subgraph of
  $G$.

  Now, we prove that $H$ is a bipartite Helly graph. Let
  $\frac{1}{2}B'_{r_i}(x_i), i=1,\ldots,n$ be a collection of pairwise
  intersecting half-balls of $H$.  Denote by
  $\frac{1}{2}B_{r_i}(x_i), i=1,\ldots,n$ the respective half-balls of
  $G$. Suppose without loss of generality that
  $\frac{1}{2}B_{r_i}(x_i)=B_{r_i}(x_i)\cap X$, where $X$ and $Y$ are
  the color classes of $G$.  Let $\frac{1}{2}B_r(u)=B_r(u)\cap X$.  We
  assert that for any $i=1,\ldots,n$, we have
  $\frac{1}{2}B_{r_i}(x_i)\cap \frac{1}{2}B_r(u)\ne
  \varnothing$. Indeed, if $x_i$ belongs to $X$, then
  $x_i\in \frac{1}{2}B_{r_i}(x_i)$ and since $x_i\in B_{r}(x)$ we also
  have $x_i\in \frac{1}{2}B_{r}(u)$. Otherwise, if $x_i\in Y$, then
  $r_i\ge 1$ and any neighbor $z$ of $x_i$ in $I(x_i,u)$ belongs to
  $\frac{1}{2}B_{r_i}(x_i)$ and to $B_r(u)$ and thus to
  $\frac{1}{2}B_r(u)$. Consequently,
  $\frac{1}{2}B_r(u),\frac{1}{2}B_{r_1}(x_1),\ldots,\frac{1}{2}B_{r_n}(x_n)$
  is a collection of pairwise intersecting half-balls of $G$. By the
  Helly property, these half-balls contain a common vertex $y$. Since
  $y\in \frac{1}{2}B_r(u)\subset B_r(u)$, $y$ is a common vertex of
  the half-balls
  $\frac{1}{2}B'_{r_1}(x_1),\ldots,\frac{1}{2}B'_{r_n}(x_n)$ of $H$.
\end{proof}

For a vertex $u$ of a bipartite Helly graph we define a graph $B_u$ in
the following way: the vertices of $B_u$ is the set $B_2(u)$ of all
vertices at distance at most 2 from $u$ and two such vertices $v,v'$
are adjacent in $B_u$ if and only if $u \in I(v,v')$.  We call $B_u$
the \emph{local graph} of the vertex $u$.

We say that a graph $G=(V,E)$ satisfies the \emph{matching-stable-set
  property} (respectively, \emph{double-matching-stable-set property})
if for any profile $\pi$ (respectively, any double-profile
$\pi = \tau^2$) on $G$,
\begin{enumerate}[(1)]
\item either there exists a vertex $z$ of $G$ such that
  $\pi(z) > \pi(N_{G}(z))$,
\item or there exists a maximal stable set $S$ such that
  $\pi(S) > \pi(N_{G}(S))$,
\item or there exists a perfect $\pi$-matching in $G$.
\end{enumerate}
Observe that for a maximal stable set $S$, we have
$S \cup N_{G}(S) = V$, thus condition (2) can be restated by requiring
that there exists a (maximal) stable set $S$ such that
$\pi(S) > \frac{1}{2}\pi(V)$.

Given a graph $G=(V,E)$ and a double-profile $\pi$ on $V$, by the
fractional Hall condition~\cite{Balinski,SchrijverA}, the graph $G$
admits a perfect $\pi$-matching if and only if for any stable set $S$
of $G$, we have $\pi(N_G(S)) \geq \pi(S)$. By this condition, for any
double-profile $\pi$ such that $G$ does not admit a perfect
$\pi$-matching, there exists a stable set $S$ such that
$\pi(N_G(S)) < \pi(S)$. Such a set is called a \emph{disabling stable
  set} for $\pi$. A family $\cD$ of stable sets of $G$ is called
\emph{disabling} for $G$ if for any double-profile $\pi$, either $G$
admits a perfect $\pi$-matching, or there exists $S \in \cD$ disabling
$\pi$.  Then, a graph $G$ satisfies the double-matching-stable-set
property if and only if the family of all maximal sets and all 1-vertex sets (i.e., singletons) of $G$ is a disabling family.

\begin{proposition}\label{prop-gen-matching-stable-set}
  A bipartite Helly graph $G$ satisfies the pairing property
  (respectively, the double-pairing property) if and only if all local
  graphs $B_u, u\in V$ satisfy the matching-stable-set property
  (respectively, the double-matching-stable-set property).
\end{proposition}

\begin{proof}
  First, let $G=(V,E)$ be a bipartite Helly graph satisfying the
  pairing property (respectively, the double-pairing property).
  Consider a vertex $u \in V$ and the corresponding local graph
  $B_u=(V_u,E_u)$ as well as an even profile (respectively, a double
  profile) $\pi$ on $V_u \subseteq V$. Note that $V_u$ is the set of
  vertices at distance at most $2$ from $u$ in $G$. Let $X=N_G(u)$ and
  $Z=\{z \in V_u: d_G(u,z)=2\}$.

  Suppose first that there exists $z \in Z$ such that
  $F_\pi(z) < F_\pi(u)$. Let $X_1 = X \cap N_G(z)$ and
  $X_2 = X\setminus X_1$. Let $Z_1 = \{z' \in Z : d_G(z,z') \leq 2\}$
  and $Z_2 = Z\setminus Z_1$. Observe that since $G$ is bipartite, for
  any $x \in X$, $xz \in E_u$ if and only if $x \in X_2$ and for any
  $z' \in Z$, $zz' \in E_u$ if and only if $z' \in Z_2$. Note also that
  $F_\pi(z) - F_\pi(u) = 2\pi(u) + 2\pi(X_2) + 2\pi(Z_2) -
  2\pi(z)$. Since $F_\pi(z) < F_\pi(u)$ and since
  $N_{B_u}(z) = \{u\}\cup X_2 \cup Z_2$, we have
  $\pi(N_{B_u}(z)) < \pi(z)$ and thus condition (1) holds.

  Suppose now that there exists $x \in X$ such that
  $F_\pi(x) < F_\pi(u)$. Let $Z_1 = Z \cap N_G(x)$ and
  $Z_2 = Z \setminus Z_1$. Observe that in $B_u$, $S = Z_1\cup \{x\}$
  is a stable set. Note that
  $F_\pi(x) - F_\pi(u) = \pi(u) - \pi(x) + \sum_{x' \in X\setminus
    \{x\}} \pi(x') - \sum_{z \in Z_1} \pi(z) + \sum_{z \in Z_2} \pi(z)
  = \pi(\{u\}\cup (X\setminus\{x\})\cup Z_2) - \pi(\{x\}\cup
  Z_1)$. Since $F_\pi(x) < F_\pi(u)$, we have
  $\pi(S) > \pi(V_u \setminus S) = \pi(N_{B_u}(S))$ and thus condition
  (2) holds. Thus, we can assume that $u$ is a local median of $\pi$
  in $G^2$. Since bipartite Helly graphs have $G^2$-connected
  medians~\cite[Proposition 62]{GpConMed}, $u \in \Med_G(\pi)$.  Since
  $G$ satisfies the pairing property (respectively, the double-pairing
  property), there exists a pairing $P$ of $\pi$ such that
  $\Med_G(\pi) = \bigcap_{\{a,b\}\in P}I(a,b)$. Hence for any
  $\{a,b\} \in P$, we have $u \in I(a,b)$ and thus $ab \in
  E_u$. Consequently, in this case, the pairing $P$ defines a
  perfect $\pi$-matching in $B_u$ and condition (3) holds.

  Assume now that for any $u \in V$, the local graph $B_u=(V_u,E_u)$
  satisfies the matching-stable-set property (respectively, the
  double-matching-stable-set property), and consider an even profile
  (respectively, a double-profile) $\pi$ on $G$. Consider a vertex
  $u \in \Med_G(\pi)$.  For each vertex $v \in \pi$ such that
  $d_G(u,v)\geq 2$, by
  Theorem~\ref{absolute-retracts}(\ref{th-bar-4}), there exists a
  vertex $z_v \in I(u,v)$ such that $d_G(u,z_v)=2$ and
  $z_v \sim I(u,v) \cap N_G(u)$.  We construct a profile $\pi'$ on
  $V_u = B_2(u,G)$ by replacing each occurrence of $v$ in $\pi$ with
  $d_G(u,v)\geq 2$ by such a $z_v$ (we keep in $\pi'$ each occurrence
  of $u$ in $\pi$ and each occurrence of $x\in N_G(u)$ in
  $\pi$). Observe that $\pi'$ is an even profile and that we can
  assume that $\pi'$ is a double-profile when $\pi$ is a double
  profile.  Note that if $u \in I(v,v')$, then $u \in
  I(z_v,z_v')$. Conversely, if $u \in I(z_v,z_v')$, then
  $u \in I(v,v')$. Indeed, since $G$ is bipartite, either
  $d_G(v,v') = d_G(v,u) + d_G(u,v')$, or
  $d_G(v,v') \leq d_G(v,u) + d_G(u,v') -2$. In the second case,
  consider the balls $B_1(u)$ of $G$ and the balls of radius
  $d_G(u,v)-1$ and $d_G(u,v')-1$ respectively centered at $v$ and
  $v'$. Since these three balls pairwise intersect, there exists a
  neighbor $x$ of $u$ such that $x \in I(u,v) \cap I(u,v')$. This
  implies that $x \sim z_v,z_{v'}$ by the definition of $z_v$ and
  $z_{v'}$, and thus $u \notin I(z_v,z_{v'})$.

  Since $u$ is a median of $\pi$ in $G$, $u$ is also a median of
  $\pi'$ in $G$. If $\pi'(u) \geq \pi'(V_u\setminus\{u\})$, then
  $\pi(u) \geq \pi(V\setminus \{u\})$ and in this case, there exists a
  pairing $P$ of $\pi$ such that for each $\{a,b\}\in P$, $a=u$ or
  $b=u$. In this case, $u\in I(a,b)$ for each $\{a,b\}\in P$, and thus
  $\Med_G(\pi) = \bigcap_{\{a,b\}\in P} I(a,b)$. Assume now that
  $\pi'(u) < \pi'(V\setminus \{u\})$.

  Suppose first that there exists a stable set $S$ of $B_u$ such that
  $\pi'(S) > \pi'(V_u\setminus S)$. Since $u$ is a universal vertex in
  $B_u$ and since $\pi'(u) < \pi'(V\setminus \{u\})$, necessarily
  $u \notin S$. Let $X = N_G(u)$ and $Z = \{z: d_G(u,z)=2\}$. Since
  for all distinct $x,x' \in X$, we have $xx' \in E_u$, we have
  $|X\cap S| \leq 1$. If $X\cap S = \{x\}$ and
  $\pi'(x) > \pi'(V_u\setminus \{x\})$, then
  $\pi(x) > \pi(V \setminus \{x\})$, and in this case,
  $\Med_G(\pi)= \{x\}$, contradicting the choice of $u$. Consequently,
  if there exists $x \in X\cap S$, then
  $S\setminus \{x\} \neq \varnothing$ and
  $S\setminus \{x\} \subseteq Z$. Moreover, if there exists
  $x \in X\cap S$, for each $z\in Z\cap S$, $d_G(x,z) = 1$ and thus
  $Z \cap X \subseteq N_G(x)$.  Since $S$ is a stable set in $B_u$,
  for all distinct $z,z' \in S \cap Z$, we have $d_G(z,z') =
  2$. Consequently, by considering all balls of radius $1$ centered at
  the vertices of $Z$, we get that there exists $x \in X$ such that
  $S \subseteq N_G(x)$. Therefore, in any case, there exists $x$ such
  that $S \subseteq \{x\}\cup N_G(x)$ and since $\{x\}\cup N_G(x)$ is
  a stable set of $B_u$, we can assume that $S= \{x\}\cup N_G(x)$. In
  this case,
  $F_{\pi'}(x) - F_{\pi'}(u) = \pi'(V_u\setminus S) - \pi'(S) < 0$,
  contradicting the fact that $u \in \Med_G(\pi')$.

  Suppose now that there exists $z \in V_u$ such that
  $\pi'(z)>\pi'(N_{B_u}(z))$. If $z \in X$, let
  $S = \{z\}\cup (N_G(z)\cap Z)$. Observe that $S$ is a stable set of
  $B_u$ and that $N_{B_u}(z) = V_u\setminus S$. Consequently, there
  exists a stable set $S$ of $B_u$ such that
  $\pi'(S) > \pi'(V_u\setminus S)$, but we already know that this is
  impossible. We can thus assume that $z \in Z$. Observe that for each
  $y \in V_u\setminus \{z\}$, either $d_G(u,y)=d_G(z,y)$, or
  $y \in N_{B_u}(z)$ and $d_G(z,y) = d_G(u,y)+2$. Consequently,
  $F_{\pi'}(z)-F_{\pi'}(u) = 2\pi'(N_{B_u}(z)) - 2\pi'(z) < 0$,
  contradicting the fact that $u \in \Med_{G}(\pi')$.

  Consequently, we can assume that there exists a $\pi'$-matching $M$
  in $B_u$.  From such a matching $M$, one can obtain a pairing $P$ of
  $\pi$ such that if $\{v,v'\} \in P$, then ${z_v}{z_{v'}} \in
  M$. Consequently, for each $\{v,v'\} \in P$, we have
  $u \in I(z_v,z_{v'})$ and $u \in I(v,v')$. Therefore we have
  $u \in \bigcap_{\{v,v'\}\in P} I(v,v') = \Med_G(\pi)$. This shows
  that $G$ satisfies the pairing property (respectively, the
  double-pairing property).
\end{proof}

The following result follows from Lemma~\ref{lem-boules-Helly} and the
proof of Proposition~\ref{prop-gen-matching-stable-set}:

\begin{corollary}\label{local-to-global}
  A bipartite Helly graph $G$ satisfies the pairing property
  (respectively, the double-pairing property) if and only if all
  bipartite Helly graphs induced by the balls of radius 2 of $G$
  satisfy the pairing property (respectively, the double-pairing
  property).
\end{corollary}

\subsection{Recognizing graphs with the double-pairing property}

We were not able to settle the complexity of deciding if a graph has
the double-pairing property. However, we can show that this problem is
in co-NP. As explained above, if we are given a profile $\pi$, one can
check whether $\pi$ admits a perfect pairing in time that is
polynomial in the size of $G$ and the size of the profile $\pi$. In
order to prove that recognizing graphs satisfying the pairing property
(respectively, the double-pairing property) is in co-NP, it would then
be sufficient to show that when a graph $G$ does not satisfy the
pairing property (respectively, the double-pairing property), there
exists an even profile $\pi$ (respectively, a double-profile $\pi$)
that does not admit a perfect pairing and whose size is polynomial in
the size of $G$. Unfortunately, we were not able to prove that these
profiles of polynomial size always exist. For the double-pairing
property, we establish the result by reformulating the problem as a
problem of inclusion of polytopes.

For a graph $G=(V,E)$ and a vertex $u\in V$, we define two polytopes
$\Ma(u)$ and $\Me(u)$ as follows. The polytope $\Me(u)$ consists of
all weight functions $b:V \to {\mathbb R}^+$ such that $u$ belongs to
the median set $\Med(b)$. In particular, for each profile $\pi$ on $G$
such that $u \in \Med(\pi)$, the weight function defined by $\pi$
belongs to $\Me(u)$.
\[
  \Me({u}) \quad \begin{cases}
    \sum_{w \in V} b(w) (d(v,w) - d(u,w))  \geq 0  &  \text{for all } v \in V,\\
    b(v) \geq 0 & \text{for all } v \in V.\\
  \end{cases}
\]
Note that $\Me({u})$ is a non-empty polytope defined by a linear
number of inequalities.

The second polytope $\Ma(u)$ consists of all weight functions
$b:V \to {\mathbb R}^+$ such that the auxiliary graph $A_u$ admits a
fractional perfect $b$-matching. The description of the polytope is
defined using the inequalities given by the fractional Hall condition.
\[
  \Ma({u}) \quad \begin{cases}
    \sum_{v \in N(S)} b(v) - \sum_{v \in S} b(v) \geq 0 &  \text{for all stable
      sets } S \text{ of } A_u=(V,E_u),\\
    b(v) \geq 0 & \text{for all } v \in V.\\
  \end{cases}
\]
The number of inequalities defining $\Ma({u})$ is potentially
exponential in the size of $G$. Note however that the separation
problem on $\Ma(u)$ can be solved in polynomial time by solving a
fractional $b$-matching problem. Such an algorithm either provides a
$b$-matching and $b$ is a point of $\Ma(u)$, or it enables to compute
a set $S$ such that $b$ violates the constraint corresponding to $S$.

Observe that both polytopes $\Me(u)$ and $\Ma(u)$ contain the origin
$\bzero$ and that they are invariant by multiplication by a positive
scalar. Therefore both $\Me(u)$ and $\Ma(u)$ are convex cones. Since
both polytopes are defined by a finite set of constraints, they are
polyhedral cones. Since these constraints have integral coefficients,
$\Me(u)$ and $\Ma(u)$ are both rational cones. For $C \in \{ \Me(u),
\Ma(u)\}$, this means that there exists a finite set of vectors $b_1,
\ldots, b_k \in \Z^{|V|}$ such that $C = \left\{ a_1 b_1 + \cdots +
  a_k b_k : a_i \in \R_+ \right\}$. Since all points of
$C$ have non-negative coordinates, necessarily $b_1, \ldots, b_k \in
\N^{|V|}$.

\begin{lemma}\label{lem-MainMe}
  Given a graph $G=(V,E)$, for any vertex $u \in V$,
  $\Ma(u) \subseteq \Me(u)$.
\end{lemma}

\begin{proof}
  For any weight function $b: V \to \R^+$, if $b$ is a point of
  $\Ma(u)$, then by the fractional Hall Condition, there exists a
  perfect $b$-matching in $A_u$ and thus, by
  Lemma~\ref{lem-bMatchMed}, $u \in \Med(b)$. This establishes that
  $b \in \Me(u)$.
\end{proof}

\begin{proposition}\label{lem-MaEqMe}
  A graph $G=(V,E)$ satisfies the double-pairing property if and only
  if $\Ma(u) = \Me(u)$ for all $u \in V$.
\end{proposition}

\begin{proof}
  First suppose that $G$ satisfies the double-pairing property.  By
  Lemma~\ref{lem-MainMe}, $\Ma(u) \subseteq \Me(u)$. We now prove the
  converse inclusion. Since $\Me(u)$ is a rational cone, there exists
  a finite set of vectors $b_1, \ldots, b_k \in \N^{|v|}$ such that
  $\Me(u) = \left\{ a_1 b_1 + \cdots + a_k b_k : a_i \in \R_+
  \right\}$. To prove $\Me(u) \subseteq \Ma(u)$, it is sufficient to
  show that $b_i \in \Ma(u)$, for each $1 \leq i \leq k$. Since both
  $\Ma(u)$ and $\Me(u)$ are invariant by multiplication by a positive
  scalar, we can assume that each coordinate of $b_i$ is an even
  integer, i.e., that $b_i$ corresponds to a double profile $\pi$.
  Since $G$ satisfies the double-pairing property, for any such double
  profile $\pi$, there exists a perfect $\pi$-matching in $A_u$, and
  thus $\pi \in \Ma(u)$.

  Conversely, assume that $\Me(u) \subseteq \Ma(u)$ for all $u \in V$
  and consider a double-profile $\tau = \pi^2$ on $G$. Pick any
  $u \in \Med(\tau) = \Med(\pi)$. Then
  $\pi \in \Me(u) \subseteq \Ma(u)$, and thus, $\dM(\pi,u)$ is
  non-empty. By Lemma~\ref{lem-integral-k}, $\tau = \pi^2$ admits a
  perfect pairing. This establishes that $G$ satisfies the
  double-pairing property.
\end{proof}

We now explain how to decide if a given graph $G$ does not satisfy the
double-pairing property in non-deterministic polynomial time.  To do
so, we first guess a vertex $u$ such that $\Ma(u) \subsetneq
\Me(u)$. We construct the auxiliary graph $A_u$ and we guess a stable
set $S$ of $A_u$ such that the corresponding constraint of $\Ma(u)$
separates a point of $\Me(u)$ from $\Ma(u)$. In order to check that
this constraint indeed separates a point of $\Me(u)$ from $\Ma(u)$, we
minimize the function
${(b(v))}_{v\in V} \mapsto \sum_{v \in N(S)} b(v) - \sum_{v \in S} b(v)$
on the polytope $\Me(u)$. This can be done in polynomial time since
this is a linear program with a linear number of constraints. If the
minimum is negative, then $\Me(u)$ is not contained in $\Ma(u)$ and
thus $G$ does not sastisfy the double-pairing property by
Proposition~\ref{lem-MaEqMe}. We thus have the following result.

\begin{proposition}
  Recognizing graphs satisfying the double-pairing property belongs to
  co-NP.
\end{proposition}

\begin{remark}
  We do not know if we can extend this method for the recognition of
  graphs with the pairing property. This is due to the fact that the
  set of profiles that admit a perfect pairing is not stable by convex
  combinations.  Indeed, consider the graph on the left
  Figure~\ref{fig-cex-PP} that is made of two disjoint triangles
  $u,v,w$ and $u',v',w'$, and denote it by $C$. Let $B$ be the
  complement of $C$ and consider the bipartite Helly graph
  $R({(\cC(B))}^*)$ defined in Subsection~\ref{subs-bipHwoPP}. Then,
  since $C$ does not admit a perfect matching, the even profile
  $\pi = (u,v,w,u',v',w')$ does not admit a perfect pairing in
  $R({(\cC(B))}^*)$. However, both the empty profile and the double
  profile $\pi^2$ admit a perfect pairing in $R({(\cC(B))}^*)$.
\end{remark}

\begin{question}
  Can we recognize graphs satisfying the double-pairing property
  (respectively, the pairing property) in polynomial time?  Is the
  recognition of graphs satisfying the pairing property in co-NP?
\end{question}

\section{Benzenoid graphs}

The main tools in our paper are the axioms (A), (B), (C), and (T),
satisfied by the median function $\Med$ in all graphs. These axioms
are simple to state and are meaningful in the context of consensus
theory. However, these axioms are difficult to handle and do not allow
to establish even simple structural properties of ABC- or ABCT-graphs.
In this section, we formulate two other axioms (T$_2$) and (E$_k$)
satisfied by $\Med$. We show that the 6-cycle $C_6$ has two ABC-functions.  Finally, we
prove that benzenoids (which are obtained from 6-cycles by
amalgamations) are ABCT$_2$- and ABCE$_2$-graphs.

\subsection{Axioms for equilateral metric triangles}
As for axiom (T), the axioms we consider deal with triplets of
vertices.  First, we start with a generalizations of axioms (T) and
(T$^-$):
\begin{enumerate}[(T$_2$)]
\item for any equilateral metric triangle $uvw$ of size 2,
  $\{u,v,w\} \subseteq L(u,v,w)$.
\end{enumerate}
Unfortunately, there is no straightforward way to generalize (T) and
(T$_2$) to triplets of vertices at distance $k$. However, we can
define the axiom (E$_k$), which coincides with (T$^-$) when $k=1$:
\begin{enumerate}[(E$_k$)]
\item \emph{Equilateral:} for any equilateral metric triangle $uvw$ of
  size $k$,
if $u \in L(u,v,w)$, then $\{u,v,w\} \subseteq L(u,v,w)$.
\end{enumerate}

\begin{lemma}\label{TE}
  The median function $\Med$ satisfies the axioms (T$_2$) and (E$_k$).
\end{lemma}

\begin{proof}
  To prove (T$_2$), pick any equilateral metric triangle $uvw$ of size
  2. Then $F_{\pi}(u)=F_{\pi}(v)=F_{\pi}(w)=4$. Any vertex $z$ of $G$
  is not adjacent to at least one of the vertices $u,v,w$, thus
  $F_{\pi}(z)\ge 4$, proving that $\{u,v,w\} \subseteq \Med(u,v,w)$.
  To prove (E$_k$), pick any equilateral metric triangle $uvw$ of size
  $k$ and suppose that $u \in \Med(u,v,w)$. Since
  $F_{\pi}(u)=2k=F_{\pi}(v)=F_{\pi}(w)$, we conclude that
  $v,w\in \Med(u,v,w)$.
\end{proof}

\subsection{$C_6$ has two  ABC-functions}
In this subsection, we define an ABC-function $L_6$ on $C_6$ different
from $\Med$. Denote by $V=\{v_0, v_1, v_2, v_3, v_4, v_5\}$ the
ordered set of vertices of the graph $C_6$ such that $v_i$ is a
neighbor of $v_{i-1}$ and $v_{i+1}$ (all additions are done modulo 6).
For a profile $\pi$ on $C_6$, let $\pi_i=\pi(v_i)$ denote the number
of occurrences of $v_i$ in $\pi$. 

Since $I(v_i, v_{i+3})=V$, for each ABC-function $L$ on $C_6$, we have
$L(v_i, v_{i+3})=I(v_i, v_{i+3})=V$ and
$L(\pi,v_i, v_{i+3}) = L(\pi)$. For each profile $\pi$, we denote by
$\pimi$ the profile such that $\pimi_i=\pi_i-\min\{\pi_i, \pi_{i+3}\}$
for each $i\in \{0,\ldots, 5\}$. Notice that either $\pimi_i$ or
$\pimi_{i+3}$ is equal to 0 and that $L(\pi^\circ) = L(\pi)$.  If
$\pimi_i, \pimi_{i+2}, \pimi_{i+4}$ are not equal to 0, then we call
$\pimi$ an \emph{alternate} profile.

Consider the following consensus function $L_6$:
\[
  L_6(\pi) =
  \begin{cases}
    \{v_i\} & \text{if } \pimi_i,\pimi_{i+2},\pimi_{i+4}>0 \text{
      and } i=\min\{j : \pimi_j=\max\{\pimi_i, \pimi_{i+2},
    \pimi_{i+4} \}\}
    \\
    \Med(\pi) & \text{otherwise.}
  \end{cases}
\]

Note that by definition of $L$, if $\pi^\circ$ is not an alternate
profile, then
$L_6(\pi) = \Med(\pi) = \Med(\pi^\circ) = L_6(\pi^\circ)$. Obviously,
$L_6\ne \Med$ since for the profile $\pi=(v_0,v_2,v_4)$ we have
$\Med(\pi)=\{v_0, v_2, v_4\}$ whereas $L_6(\pi)=\{v_0\}$. To prove
that $L_6$ is an ABC-function, we use the following lemmas:
\begin{lemma}\label{lem:c6:intmin}
  For any $\pi,\rho\in V^*$, if $\sigma=\pi\rho$ and
  $\tau=\pimi\rhom$, then $\sigm=\tau^{\circ}$.
\end{lemma}
\begin{proof}
  For any $i\in \{0,\ldots,5\}$, we have:
  \begin{align*}
    \tau_i^{\circ} &=\tau_i-\min\{\tau_i, \tau_{i+3} \}
                     =\pimi_i+\rhom_i - \min\{\pimi_i +\rhom_i,
                     \pimi_{i+3} +\rhom_{i+3}\}\\
&= \pimi_i+\rhom_i - (\min\{\pi_i + \rho_i,
                     \pi_{i+3}+\rho_{i+3}\}-\min\{\pi_i, \pi_{i+3}\} -
                     \min\{\rho_i, \rho_{i+3}  \})\\
                   &=\pi_i + \rho_i -
                     \min\{\pi_i+\rho_i,\pi_{i+3}+\rho_{i+3}\}
                     =\sigma_i-\min\{\sigma_i, \sigma_{i+3} \} =\sigm_i,
  \end{align*}
  concluding the proof.
\end{proof}

\begin{lemma}\label{lem:abcc6}
  Let $L$ be an ABC-function on $C_6$ and $\pi$ a profile such that
  $\pi^\circ$ is nonempty and non-alternate. A vertex $v_i$ belongs to
  $L(\pi) = L(\pi^{\circ})$ if and only if up to symmetry we are in the
  following cases of $\pimi$ and of $L(\pi)$:
  \begin{align}
    \text{if } \pimi &=(v_i^{\pimi_i})  & & &  L(\pi) &=\{v_i\} \\
    \text{if }  \pimi &=(v_i^{\pimi_i}, v_{i+1}^{\pimi_{i+1}})
                                        &\text{and } &\pimi_i \geq \pimi_{i+1},
                                            &  L(\pi) &=
                                                        \begin{cases}
                                                          I(v_i,v_{i+1})& \text{if } \pimi_i=\pimi_{i+1}, \\
                                                          \{v_i\}& \text{otherwise} \\
                                                        \end{cases}\\
    \text{if }  \pimi &=(v_i^{\pimi_i}, v_{i+2}^{\pimi_{i+2}}) &\text{and } &\pimi_i \geq \pimi_{i+2},
                                            &  L(\pi) &=
                                                        \begin{cases}
                                                          I(v_i,v_{i+2})& \text{if }\pimi_i=\pimi_{i+2}, \\
                                                          \{v_i\}& \text{otherwise} \\
                                                        \end{cases} \\
    \text{if }  \pimi &=(v_{i-1}^{\pimi_{i-1}},v_{i+1}^{\pimi_{i+1}})
                                        &\text{and } &\pimi_{i+1} =\pimi_{i-1},
                                            &  L(\pi) &=I(v_{i+1},v_{i-1})\\
    \text{if }    \pimi &=(v_i^{\pimi_i}, v_{i+1}^{\pimi_{i+1}},
                          v_{i+2}^{\pimi_{i+2}})
                                        & \text{and } &\pimi_i\geq \pimi_{i+1}+\pimi_{i+2},
                                            &  L(\pi) &=
                                                        \begin{cases}
                                                          I(v_i,v_{i+1})& \text{if }\pimi_i=\pimi_{i+1}+\pimi_{i+2}, \\
                                                          \{v_i\}& \text{otherwise} \\
                                                        \end{cases} \\
    \text{if }     \pimi &=(v_{i-1}^{\pimi_{i-1}},v_i^{\pimi_{i}},
                           v_{i+1}^{\pimi_{i+1}})
                                        & \text{and } &\pimi_{i} \geq |\pimi_{i+1} - \pimi_{i-1}|,
                                            &  L(\pi) &=
                                                        \begin{cases}
                                                          I(v_i,v_{i+1})& \text{if }\pimi_{i}+\pimi_{i-1}=\pimi_{i+1}, \\
                                                          I(v_i,v_{i-1})& \text{if }\pimi_{i}+\pimi_{i+1}=\pimi_{i-1}, \\
                                                          \{v_i\}& \text{otherwise}. \\
                                                        \end{cases}
  \end{align}
\end{lemma}

\begin{proof}
  Since $\pi^{\circ}$ is non-alternate, we can assume without loss of
  generality that $\pi_3^\circ = \pi_4^\circ = \pi_5^\circ = 0$. Let
  $a = \pi_0^\circ$, $b = \pi_1^\circ$, $c = \pi_2^\circ$ and up to
  symmetry, we can assume that $a \geq c$. Note that since $\pi^\circ$
  is not empty, $a+b+c \geq 1$. If $a = c$ and $b = 0$ (in this case,
  $a = c \geq 1$), then
  $L(\pi) = L(\pi^{\circ}) = L(v_0,v_2) = I(v_0,v_2)$, and for
  $v_0, v_2$, we are in Case $(iii)$ while for $v_1$, we are in Case
  $(iv)$. If $a = c$ and $b > 0$, then since
  $b \in L(v_0^a,v_2^c) = I(v_0,v_2)$, either
  $L(\pi^\circ) = L(v_1^b) = \{v_1\}$ if $a = c = 0$, or
  $L(\pi^\circ) = L(v_0^a,v_2^c) \cap L(v_1^b) = \{v_1\}$ if
  $a = c > 1$. In the first case, we are in case $(i)$, while in the
  second case, we are in case $(vi)$. Suppose now that $a > c$ and let
  $a' = a-c = |a -c|$. If $c =0$, then
  $L(\pi^{\circ}) = L(v_0^{a'},v_1^b)$. If $c > 0$, then
  $L(\pi^{\circ}) = L(v_0^c,v_2^c)\cap L(v_0^{a'},v_1^b) =
  L(v_0^{a'},v_1^b)$ since $L(v_0^c,v_2^c) = I(v_0,v_2)$ and since
  $L(v_0^{a'},v_1^b) \subseteq \{v_0,v_1\} \subseteq I(v_0,v_2)$. If
  $a' = b$ (i.e., if $a = b+c$), then
  $L(\pi^\circ) = L(v_0^{a'},v_1^b) = \{v_0,v_1\}$. In this case, if
  $c=0$, then we are in Case $(ii)$ for $v_0$ and $v_1$, and if $c>0$,
  then we are in Case $(v)$ for $v_0$ and in Case $(vi)$ for $v_1$.
  If $a' < b$ (i.e., if $b > |a -c|$), then
  $L(\pi^\circ) = L(v_0^{a'},v_1^b) = \{v_1\}$ and we are in Case
  $(ii)$ if $c=0$ and in Case $(vi)$ if $c>0$.  If $a' > b$ (i.e., if
  $a > c$), then $L(\pi^\circ) = L(v_0^{a'}) = \{v_0\}$.  Then, we are
  in Case $(i)$ if $b = c = 0$, in Case $(ii)$ if $b > c = 0$, in Case
  $(iii)$ if $c > b =0$, and in Case $(v)$ if $b>0$ and $c > 0$.
\end{proof}

The next corollary follows from previous lemma and will be used next:

\begin{corollary}\label{cor:not-alt}
  If $\pi^\circ$ is a non-alternate profile, then for
  any ABC-function $L$, $L(\pi) = \Med(\pi)$.
\end{corollary}

We now establish the main result of this subsection:

\begin{proposition}\label{th:L6-ABC}
  $L_6$ is an ABC-function on $C_6$.
\end{proposition}
\begin{proof}
  Obviously, $L_6$ satisfies the axioms (A) and (B). It remains to
  show that $L_6$ also satisfies (C). Let $\pi=\rho\sigma$ for
  $\rho,\sigma\in V^*$.  We assert that if
  $L_6(\rho)\cap L_6(\sigma) \neq \varnothing$, then
  $L_6(\rho\sigma)=L_6(\rho)\cap L_6(\sigma)$.

  Suppose that
  $v_i \in L_6(\rho)\cap L_6(\sigma) = L_6(\rhom) \cap L_6(\sigm)$.
  By the definition of $L_6$ and Lemma~\ref{lem:abcc6}, this implies
  that $\sigm_{i+3} = \rhom_{i+3} = 0$. Consequently, either $\sigm$
  (respectively, $\rhom$) is not an alternate profile, or
  $\sigm_i, \sigm_{i+2}, \sigm_{i+4} > 0$ (respectively,
  $\rhom_i, \rhom_{i+2}, \rhom_{i+4} > 0$). By
  Lemma~\ref{lem:c6:intmin}, we have $\pimi_{i+3} = 0$ and
  consequently, either $\pimi$ is not an alternate profile, or
  $\pimi_{i}, \pimi_{i+2}, \pimi_{i+4} > 0$.

  If $\rhom$ and $\sigm$ are both alternate profiles, then $\pimi$ is
  an alternate profile and for $j \in \{i+2, i+4\}$, we have
  $\pimi_i = \rhom_i+\sigm_i \geq \rhom_{j}+\sigm_{j}$ where the
  equality holds only if $\rhom_i = \rhom_{j}$ and
  $\sigm_i = \sigm_{j}$. Note that if the equality holds, then
  $0 \leq i < j \leq 5$, since $L_6(\rho) = \{v_i\}$.  Consequently,
  either $\pimi_i > \pimi_j$ or $\pimi_i = \pimi_j$ and
  $0 \leq i < j \leq 5$. Therefore we have
  $L_6(\pi) = \{v_i\} = L_6(\rho)\cap L_6(\sigma)$ and we are done.

  If neither $\rhom$ nor $\sigm$ are alternate profiles, then
  $L_6(\rho) = L_6(\rhom) = \Med(\rho)$ and
  $L_6(\sigma) = L_6(\sigm) = \Med(\sigma)$. In this case,
  $L_6(\rho)\cap L_6(\sigma) = \Med(\rho) \cap \Med(\sigma) =
  \Med(\rho\sigma) = \Med(\pi)$. If $\pimi$ is not an alternate
  profile, then $L_6(\pi) = L_6(\pimi) = \Med(\pimi) = \Med(\pi)$ and
  we are done. If $\pimi$ is an alternate profile, then
  $\pimi_i, \pimi_{i+2}, \pimi_{i+4} > 0$. Since $\rhom$ is not an
  alternate profile and since $v_i \in L_6(\rhom)$, by
  Lemma~\ref{lem:c6:intmin}, either $\rhom_{i+2}=0$ or
  $\rhom_{i+4}=0$. Similarly, either $\sigm_{i+2} = 0$ or
  $\sigm_{i+4}=0$. Since
  $\rhom_{i+2} + \sigm_{i+2} \geq \pimi_{i+2} >0 $ and
  $\rhom_{i+4} + \sigm_{i+4} \geq \pimi_{i+4}>0$, we can assume
  without loss of generality that $\rhom_{i+2} > 0$, $\rhom_{i+4}= 0$,
  $\sigm_{i+2} =0$, and $\sigm_{i+4} > 0$. By Lemma~\ref{lem:abcc6},
  we can thus assume that $\rhom = (v_i^a,v_{i+1}^b,v_{i+2}^c)$ where
  $a >0, c > 0, b\geq 0$, and $a \geq b+c$, and that
  $\sigm = (v_{i+4}^{d'},v_{i+5}^{e'},v_i^{a'})$ where
  $a' > 0, d' > 0, e' \geq 0$, and $a' \geq d'+e'$. Consequently,
  since $\pimi$ is an alternate profile,
  $\pimi = (v_i^{a+a'},v_{i+2}^{c-e'},v_{i+4}^{d'-b})$. Since
  $a \geq c > 0$ and $a' \geq d' > 0$, we have $a+a' > c-e'$ and
  $a+a' > d'-b$, and thus
  $L_6(\pi) = L_6(\pimi) = \{v_i\} = \Med(\pi)$ and we are done.

  Suppose now that $\rhom$ is an alternate profile and that $\sigm$ is
  not an alternate profile. Then
  $\rhom = (v_i^a, v_{i+2}^c, v_{i+4}^d)$ with $a \geq c > 0$ and
  $a \geq d > 0$. Moreover if $a = c$ (respectively, $a = d$), then
  $0 \leq i < i+2 \leq 5$ (respectively, $0 \leq i < i+4 \leq 5$). For
  $\sigm$, we consider the different cases given by
  Lemma~\ref{lem:abcc6}. In all cases, we show that
  $L_6(\pi) = L_6(\pimi) = \{v_i\} = L_6(\rhom) \cap L_6(\sigm)$.

  In Case $(i)$ or $(iii)$, $\sigm = (v_i^{a'},v_{i+2}^{c'})$ with
  $a' > 0$ and $a'\geq c' \geq 0$, and
  $\pimi =(v_i^{a+a'}, v_{i+2}^{c+c'}, v_{i+4}^d)$ is an alternate
  profile. Note that $a + a' > d$, that $a + a'\geq c+c'$ and that
  when $a+a' = c+c'$, we have $a = c$ and thus $0 \leq i < i+2 \leq 5$
  since $L_6(\rhom) = \{v_i\}$. In any case, we have
  $L_6(\pimi) =\{v_i\}$.

  In Case $(ii)$ or $(v)$,
  $\sigm = (v_i^{a'},v_{i+1}^{b'},v_{i+2}^{c'})$ with
  $b'> 0, c'\geq 0$, and $a' \geq b'+c' $. In this case,
  $\rhom\sigm =(v_i^{a+a'}, v_{i+1}^{b'},v_{i+2}^{c+c'},
  v_{i+4}^d)$. If $b' \geq d$, then
  $\pimi = (v_i^{a+a'}, v_{i+1}^{b'-d},v_{i+2}^{c+c'})$ and
  $L_6(\pimi) = \{v_i\}$ since
  $a + a' \geq c + b' + c' > b'- d + c + c'$. If $ b' < d$, then
  $\pimi = (v_i^{a+a'},v_{i+2}^{c+c'}, v_{i+4}^{d-b'})$ is an
  alternate profile. Since $a+a' \geq c+c'+b' > c+c'$ and
  $a+a' > d > d-b'$, we also have $L_6(\pimi) = \{v_i\}$ in this case.

  In Case $(iv)$ or $(vi)$,
  $\sigm = (v_{i-1}^{e'},v_{i}^{a'},v_{i+1}^{b'})$ with
  $a'\geq 0, b' > 0, e' > 0$, $a' \geq e'-b'$, and $a' \geq b'-e'$.
  In this case,
  $\rhom\sigm = (v_{i-1}^{e'},v_i^{a+a'}, v_{i+1}^{b'},v_{i+2}^c,
  v_{i+4}^d)$. If $d > b'$ and $c > e'$, then
  $\pimi = (v_i^{a+a'}, v_{i+2}^{c-e'}, v_{i+4}^{d-b'})$ is an
  alternate profile and $L_6(\pimi) = \{v_i\}$ since $a+a' > c-e'$ and
  $a+a' > d-b'$. If $d \leq b'$ and $c > e'$, then
  $\pimi = (v_i^{a+a'}, v_{i+1}^{b'-d},v_{i+2}^{c-e'})$ and
  $L_6(\pi) = \{v_i\}$ since $a+a' \geq c + b'-e' > b'-d + c -e'$.
  For similar reasons, we have $L_6(\pi) = \{v_i\}$ if $d > b'$ and
  $c \leq e'$. If $d \leq b'$ and $c \leq e'$, then
  $\pimi =(v_{i-1}^{e'-c},v_i^{a+a'}, v_{i+1}^{b'-d})$. Without loss
  of generality, assume that $e' - c \leq b' -d$. Let
  $\tau_1 = (v_{i-1}^{e'-c}, v_{i+1}^{e'-c})$ and
  $\tau_2 = (v_i^{a+a'}, v_{i+1}^{b'-d-e'+c})$ and observe that
  $\pimi = \tau_1\tau_2$ if $e' > c$ and $\pimi = \tau_2$ if $e' = c$.
  If $e' > c$, then $L_6(\tau_1) = I(v_{i-1},v_{i+1})$.  Since
  $a+a' \geq c + b' -e' > c+ b'-e'-d$, we have
  $L_6(\tau_2) = \{v_i\}$. Thus, in any case, we have
  $L_6(\pimi) = L(\tau_2) = \{v_i\}$.
\end{proof}

\subsection{Benzenoid graphs are ABCT$_2$- and ABCE$_2$-graphs}
By~\cite[Proposition 70]{GpConMed}, benzenoids have $G^2$-connected
medians. However they are not modular graphs (triplets of vertices at
distance 2 in each 6-cycle do not have medians). Since, by
Proposition~\ref{th:L6-ABC} $C_6$ is not an ABC-graph, benzenoids are
not ABC-graphs (as a graph family). In this subsection, we show that
they are ABCT$_2$- and ABCE$_2$-graphs.

It was shown in~\cite{Ch_benzen} that any benzenoid graph $G=(V,E)$
can be isometrically embedded into the Cartesian product
$T_1\square T_2\square T_3$ of three trees $T_1,T_2,T_3$. Namely, let
$E_1, E_2$, and $E_3$ be the edges of $G$ on the three directions of
the hexagonal grid and let $G_i=(V,E\setminus E_i)$ be the graph
obtained from $G$ by removing the edges of $E_i, i=1,2,3$. The tree
$T_i$ has the connected components of $G_i$ as the set of vertices and
two such connected components $P$ and $P'$ are adjacent in $T_i$ if
and only if there is an edge $uv\in E_i$ with one end in $P$ and
another end in $P'$.  The isometric embedding
$\varphi: V\rightarrow T_1\square T_2\square T_3$ maps any vertex $v$
of $G$ to a triplet $(v_1,v_2,v_3)$, where $v_i$ is the connected
component of $G_i$ containing $v$, $i=1,2,3$~\cite{Ch_benzen}.
The inner faces of a benzenoid $G$ are called \emph{hexagons}.
They correspond to the hexagons of the underlying hexagonal grid.
A path $P$ of  $G$ is an \emph{incomplete hexagon} if $P$
is a path of length 3 that contains an edge from each class
$E_1,E_2,E_3$ and $P$ is not included in a hexagon of $G$.

Next we show that benzenoids are ABCE$_2$-graphs.  As a warmup, we
show that $C_6$ is an ABCE$_2$-graph. By Corollary~\ref{cor:not-alt},
we only have to deal with the profiles $\pi$ such that $\pi^\circ$ is
an alternate profile. We thus need the following lemma:

\begin{lemma}\label{lem-C6-E}
  If $L$ is an ABCE$_2$-function on $C_6$, then
  $L(v_i, v_{i+2}, v_{i+4}) = \{v_i, v_{i+2}, v_{i+4}\}$.
\end{lemma}

\begin{proof}
  Let $\pi = (v_i, v_{i+2}, v_{i+4})$.  First, assume that $L(\pi)$
  contains $v_{i+1}$.  Let $\pi'=\pi v_{i+1}$. Then
  $L(\pi') = L(\pi) \cap \{v_{i+1}\} = \{v_{i+1}\}$.  Since
  $L(v_i, v_{i+2}) \cap L(v_{i+1}, v_{i+4})$ contains $v_{i+2}$, we
  have
  $L(\pi') = L(v_i,v_{i+2},v_{i+1},v_{i+4}) = L(v_i, v_{i+2}) \cap
  L(v_{i+1}, v_{i+4})$ and thus $v_{i+2} \in L(\pi')$, a
  contradiction.  For similar reasons,
  $v_{i+3},v_{i+5} \notin L(\pi)$.  Consequently,
  $L(v_i, v_{i+2}, v_{i+4}) \subseteq \{v_i, v_{i+2}, v_{i+4}\}$ and
  thus by (E$_2$), we have
  $L(v_i, v_{i+2}, v_{i+4}) = \{v_i, v_{i+2}, v_{i+4}\}$.
\end{proof}

By Corollary~\ref{cor:not-alt} and Lemma~\ref{lem-C6-E}, we obtain the
following result:

\begin{proposition}
  The 6-cycle $C_6$ is  an ABCE$_2$-graph.
\end{proposition}

\begin{proof}
  Let $\pi$ be a profile on $C_6$.  By Corollary~\ref{cor:not-alt}, we
  have to consider only the case where $\pi^\circ$ is an alternate
  profile. We assume without loss of generality that
  $\pi^\circ = (u^{k_1},v^{k_2},w^{k_3})$, where $u,v,w$ are pairwise
  at distance 2 and $1 \le k_1\le k_2 \le k_3$.  Then
  $\pi^\circ =
  (u^{k_1},v^{k_1},w^{k_1},v^{k_2-k_1},w^{k_2-k_1},w^{k_3-k_2})$.  By
  Lemma~\ref{lem-C6-E}, $L(u^{k_1},v^{k_1},w^{k_1}) = \{u,v,w\}$.  If
  $k_2 > k_1$, then $L(v^{k_2 - k_1},w^{k_2 - k_1}) = I(v,w)$, and if
  $k_3 > k_2$, then $L(w^{k_3-k_2}) = \{w\}$. Consequently,
  $L(\pi^\circ) = \{u,v,w\}$ if $k_1 = k_2 = k_3$,
  $L(\pi^\circ) = \{v,w\}$ if $k_1 < k_2 = k_3$, and
  $L(\pi^\circ) = \{w\}$ if $k_1 \leq k_2 < k_3$. Consequently, for
  any ABCE$_2$ function and any profile $\pi$, we have
  $L(\pi) = L(\pi^\circ) = \Med(\pi^\circ)= \Med(\pi)$ and $\Med$ is
  the unique ABCE$_2$ function on $C_6$.
\end{proof}

To show that benzenoids are ABCE$_2$-graphs, we will need the
following known lemma:

\begin{lemma}[{\cite[Claim~71]{GpConMed}}]\label{lem-benz-gated}
  All hexagons and incomplete hexagons in benzenoids are gated.
\end{lemma}

Now we show that it suffices to prove that benzenoids are
ABCT$_2$-graphs:

\begin{lemma}\label{lem:benz-E2-T2}
  On benzenoids, ABCE$_2$-functions are ABCT$_2$-functions.
\end{lemma}

\begin{proof}
  Let $u,v,w$ be a triplet of vertices of $G$ such that
  $d(u,v) = d(u,w) = d(v,w) = 2$, $I(u,v) \cap I(u,w) = \{u\}$,
  $I(u,v) \cap I(v,w) = \{v\}$, and $I(v,w) \cap I(w,u) = \{w\}$.
  Note that $u$, $v$, and $w$ are pairwise nonadjacent vertices of
  some hexagon $C$. We assert that $L(u,v,w)\subseteq \{u,v,w\}$.  Let
  $x \in L(u,v,w)$ and $\pi'=(x,u,v,w)$. Then
  $L(u,v,w,x) = L(u,v,w) \cap L(x)= \{x\}$.  Since $C$ is gated by
  Lemma~\ref{lem-benz-gated}, let $x'$ be the gate of $x$ in $C$.
  First assume that $x' \neq u,v,w$, say $x\sim u,v$.  Then $w$ is the
  opposite to $x'$ vertex of $C$ and $I(u,v) \subseteq L(x,w)$. Thus,
  $I(u,v) \subseteq L(u,v) \cap L(x,w) = L(\pi')$, a contradiction.
  Consequently, $x' \in \{u,v,w\}$, say $x' = u$.  Hence $u$ belongs
  to $L(x,v)$. Since $u$ in $L(u,w)$, we conclude that
  $u\in L(u,v,w,x) = \{x\}$, thus $u = x\in L(u,v,w)$. By (E$_2$),
  $\{u,v,w\} \subseteq L(u,v,w)$.
\end{proof}

\begin{lemma}\label{lem-benz-c6}
  Let $(u,v)$ be a 2-pair of a benzenoid $G$.  If there exists a
  profile $\pi$ such that $F_{\pi}$ is not pseudopeakless on $(u,v)$,
  then $u$ and $v$ belong to the same hexagon of $G$.
\end{lemma}

\begin{proof}
  Let $w$ be a common neighbor of $u$ and $v$ and let $x$ be a vertex
  such that $2d(w,x)>d(u,x)+d(v,x)$.  Since benzenoids are bipartite,
  $d(u,x)=d(v,x)=d(w,x)-1$ and since incident edges of $G$ are in
  different classes, we can suppose that $wu\in E_1$ and $wv\in
  E_2$. Let $y$ be a neighbor of $u$ in $I(u,x)$. Then $uy$ cannot
  belong to $E_2$, otherwise in $T_2$ the vertex $w_2$ would have two
  neighbors $y_2$ and $v_2$ belonging to $I(w_2, x_2)$, contrary to
  the fact that $T_2$ is a tree.  Hence $uy\in E_3$. Note that
  $(y,u,w,v)$ is not gated. By Lemma~\ref{lem-benz-gated}, $(y,u,w,v)$
  is not an incomplete hexagon, whence $u,v$ belongs to the same hexagon.
\end{proof}

The next lemma follows from the fact that every vertex of a hexagon
$C$ belongs to the interval between two opposite vertices of $C$.

\begin{lemma}\label{lem-benz-opp}
  Let $C$ be a hexagon of a benzenoid $G$ and let $\overline{x}$ be
  the vertex of $C$ opposite to the gate of a vertex $x$ in $C$. Then
  every vertex of $C$ belong to $L(x,\overline{x})$.
\end{lemma}

The next lemma follows from Lemma~\ref{lem-benz-gated}:

\begin{lemma}\label{lem-benz-Lux}
  Let $C$ be a hexagon of a benzenoid $G$ and let $(u,v)$ be a 2-pair
  of $C$. Then for any vertex $z$ of $G$ having $v$ as the gate in
  $C$, we have $u,v \in L(u,z)=I(u,z)$.
\end{lemma}

The proof of the next lemma is similar to the proof of
Lemma\ref{lem-Wuveq}:

\begin{lemma}\label{lem-benz-Luvx}
  Let $C$ be a hexagon of a benzenoid $G$ and let $u,v,w$ be three
  vertices of $C$ at pairwise distance $2$. Then for any vertex $z$ of
  $G$ that has $w$ as the gate in $C$, we have $u,v \in L(u,v,z)$.
\end{lemma}

\begin{proof}
  We start with the following claim:
  \begin{claim}\label{claim-benz-wLuvz}
    Either $u,v,w \in L(u,v,z)$ or $u,v,w \notin L(u,v,z)$.
  \end{claim}
  \begin{proof}
    The proof is similar to the proof of Claim~\ref{claim-wLuvz},
    using (T$_2$) instead of (T).  By symmetry it suffices to show
    that $u \in L(u,v,z)$ if and only if $w \in L(u,v,z)$.  Suppose
    that $\{u,w\} \cap L(u,v,z) \neq \varnothing$ and consider the
    profile $(u,v,z,u,w)$. Since $u,w \in I(u,w) = L(u,w)$ and
    $\{u,w\} \cap L(u,v,z) \neq \varnothing$, we have
    $L(u,v,z,u,w) = I(u,w) \cap L(u,v,z)$.  By (T$_2$),
    $\{u,v,w\} \subseteq L(u,v,w)$ and since $L(u,z) = I(u,z)$, we
    have
    $u,w \in L(u,v,w) \cap L(u,z) = L(u,v,z,u,w) = I(u,w) \cap
    L(u,v,z)$. Consequently, $u \in L(u,v,z)$ iff $w \in L(u,v,z)$.
    This ends the proof of the claim.
  \end{proof}

  If $u \in L(u,v,z)$ or $v\in L(u,v,z)$, by
  Claim~\ref{claim-benz-wLuvz} we are done.  If $u,v \notin L(u,v,z)$,
  then let $z' \in L(u,v,z)$.  By Lemma~\ref{lem-xinL},
  $z'\in L(u,v,z')$ and by Lemma~\ref{lem-inter-interval},
  $I(u,z')\cap I(v,z')=\{z'\}$.  Since $d(u,v)=2$ and $G$ is
  bipartite, this implies that $d(u,z')=d(v,z')$. By
  Lemma~\ref{lem-benz-gated}, every $C_6$ is gated, and then either
  $z'=w$ and we are done by Claim~\ref{claim-benz-wLuvz}, or
  $z'\sim u,v$.  We show that the second case cannot happen by
  considering the profile $(u,v,z,z')$. If $z' \sim u,v$, then
  $z' =\bar{w}$ and
  $I(u,v) = L(u,v) \subseteq L(z',w) \subseteq L(z',z)$. Consequently,
  if $z' \sim u,v$, we have $z' \in L(u,v,z) \cap L(z') = \{z'\}$ and
  $z' \in L(u,v)\cap L(z',z) = L(u,v) = I(u,v)$. This implies that
  $\{z'\} = L(u,v,z,z') = I(u,v)$, a contradiction. This ends the
  proof of the lemma.
\end{proof}

Let $C=(v_0, v_1, v_2, v_3, v_4, v_5)$ be a hexagon of $G$ and
$\pi_C^i$ be the restriction of a profile $\pi$ to the set of all
vertices of $G$ that have $v_i$ as their gate in $C$.  By convenience,
we consider every index $i$ of $v_i$ modulo $6$.

\begin{lemma}\label{lem-benz-pi'}
  For any profile $\pi$ on $G$, consider the profile
  $\pi' =\pi, v_0^{l_1+l_2+l_3+l_4}, v_2^{l_0+l_1+l_4+l_5}, v_4^{l_1}$
  where $l_i = |\pi_C^i|$. Then $v_0, v_2 \in L(\pi')$.
\end{lemma}

\begin{proof}
  Let $\pi_C^i=(x_{i,1}, x_{i,2}, \ldots ,x_{i,l_i})$ for
  $i\in \{0, 1, \ldots, 5\}$.  By Lemma~\ref{lem-benz-opp},
  $v_0,v_2 \in \bigcap_{j=1}^{l_1} L(x_{1,j}, v_4)$,
  $v_0,v_2 \in \bigcap_{j=1}^{l_3} L(x_{3,j}, v_0)$ and
  $v_0,v_2 \in \bigcap_{j=1}^{l_5} L(x_{5,j}, v_2)$.  By
  Lemma~\ref{lem-benz-Lux},
  $v_0,v_2\in \bigcap_{j=1}^{l_0} L(x_{0,j}, v_2)$ and
  $v_0,v_2 \in \bigcap_{j=1}^{l_2} L(x_{2,j}, v_0)$.  By
  Lemma~\ref{lem-benz-Luvx},
  $v_0,v_2 \in \bigcap_{j=1}^{l_4} L(x_{4,j}, v_2, v_0)$.  Observe
  also that $v_0,v_2 \in L(v_0^{l_1},v_2^{l_1})$.  Consequently,
  $v_0,v_2 \in L(\pi')$.
\end{proof}

\begin{lemma}\label{lem-benz-F}
  Let $(v_0,v_2)$ be a 2-pair of a hexagon $C=(v_0,v_1,\ldots ,v_5)$
  of a benzenoid $G$.  Then for any profile $\pi$ on $G$, the
  following holds:
  \begin{enumerate}[(1)]
  \item if $F_{\pi}(v_0) = F_{\pi}(v_2)$, then $v_0 \in L(\pi)$ iff
    $v_2 \in L(\pi)$.
  \item if $F_{\pi}(v_0) >F_{\pi}(v_2)$, then $v_0 \notin L(\pi)$.
  \end{enumerate}
\end{lemma}

\begin{proof}
  For any $0 \le i \le 5$, let $l_i = |\pi_C^i|$.  Assume that
  $F_{\pi}(v_0)\geq F_{\pi}(v_2)$ and observe that
  $F_{\pi}(v_0) - F_{\pi}(v_2) = 2(l_2+l_3) - 2 (l_0+l_5)$.  As in
  Lemma~\ref{lem-benz-pi'}, consider the profile
  $\pi'=(\pi, v_0^{l_1+l_2+l_3+l_4},v_2^{l_0+l_1+l_4+l_5},
  v_4^{l_1})$.  Suppose that $F_{\pi}(v_0) = F_{\pi}(v_2)$ (then
  $l_2+l_3+l_4=l_0+l_4+l_5=p$) and that $v_0\in L(\pi)$.  Since
  $v_0\in L(v_0^p,v_2^p)=L(v_0,v_2)$, since
  $v_0\in L(v_0^{l_1}, v_2^{l_1}, v_4^{l_1})=L(v_0,v_2,v_4)$ by
  (T$_2$), and since $v_0 \in L(\pi)$, we have
  $L(\pi') = L(\pi) \cap L(v_0^p,v_2^p) \cap L(v_0^{l_1}, v_2^{l_1},
  v_4^{l_1})$. By Lemma~\ref{lem-benz-pi'}, $v_2 \in L(\pi')$, and
  thus $v_2 \in L(\pi)$.  Suppose now that
  $F_{\pi}(v_0) > F_{\pi}(v_2)$ (then $q=l_2+l_3+l_4>l_0+l_4+l_5=r$).
  Since $v_0 \in L(v_0^r,v_2^r)=L(v_0,v_2)$,
  $v_0\in L(v_0^{l_1}, v_2^{l_1}, v_4^{l_1})=L(v_0,v_2,v_4)$ by
  (T$_2$) and $v_0 \in L(v_0^{q-r})=L(v_0)=\{v_0\}$, then
  $L(\pi') = L(\pi) \cap L(v_0^r,v_2^r) \cap L(v_0^{l_1}, v_2^{l_1},
  v_4^{l_1}) \cap L(v_0^{q-r}) = \{v_0\}$ if $v_0 \in L(\pi)$,
  contradicting Lemma~\ref{lem-benz-pi'}.
\end{proof}

\begin{theorem}
  Benzenoids are ABCT$_2$-graphs and ABCE$_2$-graphs.
\end{theorem}
\begin{proof}
  Let $L$ be an ABCT$_2$-function on a benzenoid graph $G$ and let
  $\pi\in V^*$.  We first show that $L(\pi) \subseteq
  \Med_G(\pi)$. Pick any $u \in L(\pi)$.  If $u\notin \Med(\pi)$,
  since benzenoids have $G^2$-connected medians, by
  Theorem~\ref{th-cmed-p} there exists a vertex $v$ such that
  $1\leq d(u,v) \leq 2$ and $F_{\pi}(v)<F_{\pi}(u)$.  Since $G$ is
  bipartite, $G$ satisfies the triangle condition TC.  Therefore, if
  $d(u,v)=1$, we obtain a contradiction with
  Lemma~\ref{lem-L-wconv}(2). Thus, we can suppose that $d(u,v)=2$ and
  that $F_{\pi}(w)\ge F_{\pi}(u)$ for any neighbor $w$ of $u$. By
  Lemma~\ref{lem-benz-c6} applied to the pair $(u,v)$, $u,v$ belong to
  the same hexagon and consequently by Lemma~\ref{lem-benz-F}(2), $u$
  cannot belong to $L(\pi)$, a contradiction.

  Now we show the converse inclusion $\Med(\pi)\subseteq L(\pi)$.  We
  suppose that there exist $u \in L(\pi)$ and
  $v\in \Med(\pi) \setminus L(\pi)$ minimizing $d(u,v)$.  Since
  $\Med(\pi)$ is $G^2$-connected, $d(u,v) \leq 2$.  Since
  $u, v \in \Med(\pi)$, $F_{\pi}(u) = F_{\pi}(v)$.  If $d(u,v) = 1$,
  we obtain a contradiction with Lemma~\ref{lem-L-wconv}(1).  If
  $d(u,v)=2$, by our choice of $u$ and $v$, we must have
  $F_{\pi}(w)>F_{\pi}(u)$ for any $w\in I^{\circ}(u,v)$.  Thus, by
  Lemma~\ref{lem-benz-c6}, $u$ and $v$ belong to the same $C_6$ and we
  obtain a contradiction with Lemma~\ref{lem-benz-F}. Consequently,
  benzenoids are ABCT$_2$-graphs and by Lemma~\ref{lem:benz-E2-T2},
  they are ABCE$_2$-graphs.
\end{proof}

\section{Perspectives}

In this paper, we considerably extended the classes of graphs whose
median function can be characterized by a set of simple
axioms. Namely, we proved that graphs with connected medians are
ABCT-graphs, the benzenoids graphs are ABCT$_2$-graphs, and that
modular graphs with $G^2$-connected medians are ABC-graphs. One
important subclass of modular graphs with $G^2$-connected medians are
bipartite Helly graphs. We showed that the graphs with the pairing or
double-pairing property are proper subclasses of bipartite Helly
graphs. Median graphs -- another previously known class of ABC-graphs
-- are also modular and have connected medians.

The problem of characterizing ABC-graphs remains open. In view of our
results, it is perhaps plausible (but dangerous) to conjecture that
ABC-graphs are exactly the modular graphs with $G^2$-connected
medians. However, proving that ABC-graphs are bipartite or modular was
out of reach for us (due to the weakness of the axioms (A), (B), and
(C)). We also failed to prove that $G$ is an ABC-graph or an
ABCT-graph if and only if each 2-connected component of $G$ is.  On
the opposite side, one can ask if the question of deciding if $G$ is
an ABC-graph or ABCT-graph is decidable. In~\cite{GpConMed} we showed
that a modular graph has $G^2$-connected medians if and only if this
holds for each interval $I(u,v)$ with $3\le d(u,v)\le 4$ and any
profile $\pi$ included in $I(u,v)$.  Furthermore, we presented
sufficient combinatorial conditions on 3- and 4-intervals under which
a modular graph $G$ has $G^2$-connected medians.  However, we do not
know if modular graphs with $G^2$-connected medians can be
characterized combinatorially by considering only profiles of bounded
size included in $I(u,v)$. Bridged (and, in particular, chordal
graphs) are graphs with $G^2$-connected medians. However, it is open
to us if they are ABCT-graphs.

We proved that the problem of deciding if a graph satisfies the double
pairing property is in co-NP. However we do not know if one can decide
in polynomial time (or even in non-deterministic polynomial time) if a
graph satisfies the pairing or the double-pairing property.

\subsection*{Acknowledgements} We are grateful to the referee for a
careful reading of the first version and useful comments.  This work
was partially supported by ANR project DISTANCIA
(ANR-17-17-CE40-0015).

% \bibliographystyle{plainurl}
% \bibliography{refs-medians}
\end{document}